\documentclass{article}
\usepackage{authblk}
\usepackage{amssymb}  
\usepackage[margin=2.9cm]{geometry}

\usepackage{amsmath}
\usepackage{rlepsf}
\usepackage{graphicx}
\usepackage[all]{xy}
\usepackage{amscd}\usepackage{comment}

\newtheorem{Theorem}{Theorem}[section]

\newtheorem{definition}[Theorem]{Definition}
\newtheorem{example}[Theorem]{Example}
\newtheorem{remark}[Theorem]{Remark}
\newtheorem{lemma}[Theorem]{Lemma}
\newtheorem{proposition}[Theorem]{Proposition}
\newenvironment{proof}[1][Proof]{\textbf{#1.} }{ \rule{0.5em}{0.5em}}

\def \Z {\mathbb{Z}}
\def \Gc {\mathcal{G}}
\def \d {\partial}
\def \tr {\triangleright}
\def \Ad {\mathrm{Ad}} 
\def \id {\mathrm{id}}
\def \s {\scriptstyle}
\def \R {\mathbb{R}}
\def \C {\mathcal{C}}
\def \tn {\otimes}
\def \k {\kappa}
\def \Cl {\C_{\k}}
\def \w {\omega}
\def \ra {\xrightarrow}
\def \f {\phi}
\def \p {\psi}
\def \tl {\triangleleft}
\def \N {\mathbb N}
\def \rref#1{(\ref{#1})}
\def \ad {\mathrm{ad}}
\def \PGL {\rm PGL}
\def \GL {\rm GL}
\def \de {\delta}

\begin{document}

\title{Link invariants from finite categorical groups} 
\author[1]{Jo\~{a}o Faria Martins}             
\author[2]{Roger Picken}

\affil[1]{Departamento de Matem\'{a}tica,
         Faculdade de Ci\^{e}ncias e Tecnologia,
         Universidade Nova de Lisboa,
         Quinta da Torre,
         2829-516 Caparica.
         Portugal.  
        {\it jn.martins@fct.unl.pt}    
         {\it Current address:} Department of Pure Mathematics,
School of Mathematics,
University of Leeds,
Leeds, LS2 9JT, UK. \medskip
}

\affil[2]{Center for Mathematical Analysis, Geometry and Dynamical Systems,
         Mathematics Department,
         Instituto Superior T\'ecnico, Universidade de Lisboa,
         Av. Rovisco Pais,
         1049-001 Lisboa
         Portugal.
         {\it roger.picken@tecnico.ulisboa.pt} 
         }

\maketitle

\begin{abstract}
 We define an invariant of tangles and framed tangles given a finite crossed module and a pair of functions, called a Reidemeister pair, satisfying natural properties. We give several examples of Reidemeister pairs derived from racks, quandles, rack and quandle cocycles, and central extensions of groups. We prove that our construction includes all rack and quandle cohomology (framed) link invariants, as well as the Eisermann invariant of knots. We construct a class of Reidemeister pairs which constitute a lifting of the Eisermann invariant, and show through an example that this class is strictly stronger than the Eisermann invariant itself.
 \medskip


\end{abstract}

\maketitle

\section{Introduction}
In knot theory, for a knot $K$, the fundamental group $\pi_1(C_K)$ of the knot complement $C_K$, also known as the knot group, is an important invariant, which however depends only on the homotopy type of the complement of $K$ (for which it is a complete invariant), and therefore, for example, it fails to distinguish between the square knot and the granny knot, which have homotopic, but {non-diffeomorphic} complements. Nevertheless, a powerful knot invariant $I_G$ can be defined from any finite group $G$, by counting the number of morphisms from the knot group into $G$. In a recent advance, Eisermann \cite{E2} constructed, from any finite group $G$ and any $x \in G$, an invariant $E(K)$ that is closely associated to a complete invariant \cite{W}, known as the peripheral system, consisting of the knot group $\pi_1(C_K)$ and the homotopy classes of a meridian $m$ and a longitude $l$. Eisermann gives examples showing that his invariant is capable of distinguishing mutant knots  as well as detecting chiral (non-obversible), non-inversible and non-reversible  knots (using the terminology for knot symmetries employed in \cite{E2}). Explicitly the Eisermann invariant {for a knot $K$ is:}
\[{E(K)=\sum_{\left\{{ f\colon \pi_1(C_K) \to G|  f(m)=x} \right\}} f(l),}\]
and takes values in the group algebra $\Z[G]$ of $G$.

Eisermann's invariant has in common with many other invariants that it can be calculated by summing over all the different ways of colouring knot diagrams with algebraic data. Another well-known example of such an invariant is the invariant $I_G$ defined  above, which can be calculated by counting the number of colourings of the arcs of a diagram with elements of the  group $G$, subject to certain (Wirtinger) \cite{BZ} relations at each crossing.
 Another familiar construction is to use elements of a finite quandle to colour the arcs of a diagram, satisfying suitable rules at each crossing \cite{FR,CJKLS}. Note that the fundamental quandle of the knot complement is a powerful invariant that distinguishes all knots, up to simultaneous orientation reversal of $S^3$ and of the knot (knot inversion); see  \cite{J}. Other invariants refine the notion of colouring diagrams by assigning additional algebraic data to the crossings - a significant example is quandle cohomology \cite{CJKLS}. Eisermann's invariant can be viewed in several different ways, but for our purpose the most useful way is to see it as a quandle colouring invariant using a special quandle (the ``Eisermann quandle") associated topologically with the longitude and so-called partial longitudes coming from the diagram.

A diagram $D$ of a knot or link $K$ naturally gives rise to a particular presentation of the knot group, known as the Wirtinger presentation.
Our {first} observation is that there is also a natural crossed module of groups associated to a knot diagram \cite{BHS,BM,FM2}, namely $\Pi_2(X_D,Y_D)=\big(\partial\colon \pi_2(X_D,Y_D) \to \pi_1(Y_D)\big)$  - see the next section for the definition of a crossed module of groups and the description of $\Pi_2(X_D,Y_D)$.  This crossed module is a totally free crossed module \cite{BHS}, where  $\pi_1(Y_D)$ is the free group on the arcs of $D$ and   $\Pi_2(X_D,Y_D)$ is the free crossed module on the crossings of $D$.
The crossed module  $\Pi_2(X_D,Y_D)$ is not itself a knot  invariant, although it can be related to the knot group since $\pi_1(C_K)={\rm coker}(\d)$. However, up to crossed module homotopy \cite{BHS}, $\Pi_2(X_D,Y_D)$ is a knot invariant, depending only on the homotopy type of the complement $C_K$. Therefore, given a finite crossed module $\Gc=(\partial\colon E \to G)$, one can define a knot invariant $I_\Gc$  by counting all possible  colourings of the arcs and crossings of a diagram $D$ with elements of $G$ and $E$ respectively, satisfying some natural compatibility relations (so that colourings correspond to crossed module morphisms $\Pi_2(X_D,Y_D) \to \Gc$), and then normalising \cite{FM,FM2}. 

This invariant $I_\Gc(K)$  depends only on the homotopy type of the complement $C_K$ \cite{FM2,FM3}, thus it is a function of the knot group alone. Our main insight is that imposing a suitable restriction on the type of such colourings, and then counting the possibilities, gives a finer invariant. The restriction is to colour the arcs and crossings in a manner that is a) compatible with the crossed module structure, and b) such that the assignment at each crossing is given in terms of the assignments to two incoming arcs by two functions (one for each type of crossing), termed a Reidemeister pair. Since we are choosing particular free generators of $\Pi_2(X_D,Y_D)$ this takes away the homotopy invariance of the invariant. 

The two functions making up the Reidemeister pair must satisfy some conditions, and depending on the conditions imposed, our main theorem (Theorem \ref{maintheorem}) states that one obtains in this way an invariant either of knots or of framed knots (knotted ribbons). In fact our statement extends to tangles and framed tangles. 

{This invariant turns out to have rich properties, which are described in the remainder of the paper (Section 4). It includes as special cases the invariants coming from rack and quandle colourings, from rack and quandle cohomology and the Eisermann invariant (subsections 4.1. and 4.2). In section {4.3} we introduce the notion of an Eisermann lifting, namely a Reidemeister pair derived from a {central extension of groups} which reproduces the arc colourings of the Eisermann quandle, combined with additional information on the crossings. We give a simple example of an Eisermann lifting that is strictly stronger than the Eisermann invariant  it comes from. Finally, in subsection {4.4}, we give a homotopy interpretation of the Eisermann liftings, by using the notion of {non-abelian tensor product and non-abelian} wedge product of groups, {defined by Brown and Loday {\cite{BrL0,BrL}.}}}

\section{{Crossed modules of groups  and categorical groups }}\label{Xmod}

\subsection{Definition of crossed modules and first examples}
\begin{definition}[Crossed module of groups]\label{LCM}
 A crossed module of groups, { ${\Gc= ( \d\colon E \to  G,\tr)}$,} is given by a group morphism $\d\colon E \to G$ together with a  left action $\tr$ of $G$ on $E$ by automorphisms, {such that} the following conditions (called Peiffer equations) hold:
\begin{enumerate}
 \begin{minipage}{0.5 \textwidth}
  \item {$\d(g \tr e)=g \d(e)g^{-1};\,\, \forall g \in G, \forall e \in E,$}\end{minipage}
 \begin{minipage}{0.5 \textwidth} \item {$\d(e) \tr f=efe^{-1};\,\, \forall e,f  \in E$}.\end{minipage}
\end{enumerate}
The crossed module of groups is said to be finite, if both groups $G$ and $E$ are finite. 
Morphisms of crossed modules are defined in the obvious way.
\end{definition}

\begin{example}
Any pair of finite groups $G$ and $E$, with $E$ abelian, gives a finite crossed module of groups with 
trivial $\d, \tr$ {(i.e. $\d(E)=1, \, g\tr e = e, \forall g \in G, \forall e \in E$). More generally we can choose any action of $G$ on $E$ by automorphisms, with trivial boundary map $\d\colon E \to G$.}
\label{simplestexample} 
\end{example}

\begin{example}
{Let $G$ be any finite group. Let $\Ad$ denote the adjoint action of $G$ on $G$. Then ${(\id \colon G \to G, \Ad)}$ is a finite crossed module of groups.}
\end{example}

 \begin{example}
There is a well-known construction of a crossed module of groups in algebraic topology, namely the fundamental crossed module associated to a pointed pair {$(X,Y)$ of path-connected topological spaces $(X,Y)$, thus $Y\subset X$, namely the crossed module:}
{$
\Pi_2(X,Y)= ( \d\colon \pi_2(X,Y) \to \pi_1(Y),\tr),
$}
with the obvious boundary map $\partial\colon \pi_2(X,Y) \to \pi_1(Y)$, {and the usual} action of $\pi_1(Y)$ on $\pi_2(X,Y)$; see figure \ref{action}, and  \cite{BHS} for a complete definition. This is an old result of Whitehead {\cite{W2,W3}. }
\begin{figure} 
\centerline{\relabelbox 
\epsfysize 2.5cm
\epsfbox{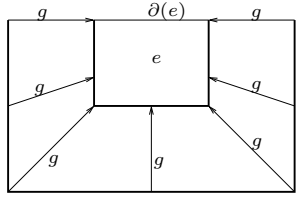}
\relabel{X}{$\s{g}$}
\relabel{Y}{$\s{g}$}
\relabel{Z}{$\s{g}$}
\relabel{W}{$\s{g}$}
\relabel{S}{$\s{g}$}
\relabel{T}{$\s{g}$}
\relabel{U}{$\s{g}$}
\relabel{e}{$\s{e}$}
\relabel{f}{$\s{\d(e)}$}
\endrelabelbox}
\caption{{\label{action} The action of an element $g \in \pi_1(Y)$ on an $e \in \pi_2(X,Y)$}.}
\end{figure}
\end{example}

\begin{example}
We may construct a topological pair $(X_D,Y_D)$ from a link diagram $D$ of a link $K$ {in $S^3$}, and thus obtain a crossed module $\Pi_2(X_D,Y_D)$, associated to the diagram. Regard the diagram as the orthogonal projection onto the $z=0$ plane in $S^3= \R^3 \cup \{\infty\}$ of a link $K_D$, isotopic to $K$, lying entirely in the plane $z=1$, except in the vicinity of each crossing point, where the undercrossing part of the link descends to height $z=-1$.
Then we take $X_D$ {(an excised link complement)} to be the link complement {$C_K$ of $K_D$ minus an open ball}, and $Y_D$ to be the $z\geq 0$ subset of $X_D$, i.e.
\[
X_D := \big (S^3 \setminus {n(K_D)}){\cap \{ (x,y,z)| z\geq -2\}}, \quad \quad Y_D := X_D \cap \{ (x,y,z)| z\geq 0\},
\]
{where $n(K_D)$ is an open regular neighbourhood of $K_D$ in $S^3$.} {Note that the space $X_D$ depends only on $K$ itself, so we can write it as $X_K$. The same is not {true for} $Y_D$.}

Each arc of the diagram $D$ corresponds to a generator of $\pi_1(Y_D)$ and there are no relations {between generators}. Each crossing of the diagram $D$ corresponds to a generator of $\pi_2(X_D, Y_D)$, namely $\epsilon: [0,1]^2\rightarrow X_D$, where the image under $\epsilon$ of the interior of $[0,1]^2$ lies entirely in the region $z<0$ and the image of the boundary of $[0,1]^2$, a loop contained in $Y_D$, encircles the crossing as in Figure \ref{epsPQRS}.
\begin{figure}
\begin{center}
\includegraphics[width=7cm]{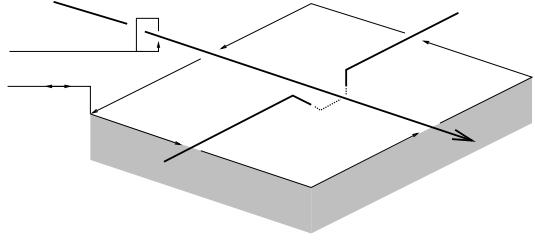}
\end{center}
\caption{\label{epsPQRS} {A generator of $\pi_1(Y_D)$ and a generator of $\pi_2(X_D, Y_D)$.}}
\end{figure}
This boundary loop is the product of four arc loops in $\pi_1(Y_D)$. 
{Again there are no {(crossed module)} relations between the generators of $\pi_2(X_D, Y_D)$ associated to the crossings}. (This  can  be justified by Whitehead's theorem \cite{W2,W3, BHS}: for {path-connected spaces} $X, \,Y$, if $X$ is obtained from $Y$ by attaching 2-cells, then $\Pi_2(X,Y)$ is the free crossed module on the attaching maps of the 2-cells. Note that $X_D$ is homotopy equivalent to the CW-complex obtained from $Y_D$ by attaching a 2-cell for each crossing).

We observe that the quotient 
$\pi_1(Y_D) / {\rm im}\, \d$ is isomorphic to the fundamental group of the link complement {$C_K=\pi_1(S^3 \setminus n(K))$} for any diagram $D$, since quotienting $\pi_1(Y_D)$ by ${\rm im}\, \d$ corresponds to imposing the Wirtinger relations \cite{BZ}, which produces the Wirtinger presentation of {$\pi_1(C_K)$}, coming from the particular choice of diagram. Thus 
$\Pi_2(X_D,Y_D)$, whilst not being itself a link invariant {(unless \cite{FM2,FM3} considered up to crossed module homotopy)}, contains an important link invariant, namely {$\pi_1(C_K)$,} by taking the above quotient. The guiding principle in the construction to follow is to extract additional Reidemeister invariant information from the crossed module $\Pi_2(X_D,Y_D)$.
\label{fundexp}
\end{example}

\subsection{A monoidal category $\C(\Gc)$ defined from a categorical group $\Gc$}\label{amc}
It is well-known that a crossed module of groups $\Gc$ gives rise to a categorical group, denoted $\mathcal{C}(\Gc)$, a monoidal groupoid where all objects and arrows are invertible, with respect to the tensor product operation; see \cite{BM,BHS,BL,FM,ML}. We recall the essential details.  Given a crossed module of groups ${\Gc= ( \d\colon E \to  G,\tr)}$, the {monoidal} category $\mathcal{C}(\Gc)$ has $G$ as its set of objects, and the morphisms from $U\in G$ to $V\in G$ are given by all pairs $(U,e)$ with $e\in E$ such that $\d(e)= VU^{-1}$. It is convenient to think of these morphisms as downward pointing arrows 
and / or  to represent them as squares - see \eqref{CGmorphisms}. 
\begin{equation}\label{CGmorphisms}
\xymatrix{
 & U \ar[d]_{(U,e) }\\ &V} \quad  \begin{CD}\\\\\textrm{ \quad \quad \quad \textrm{{or}}}\end{CD}  \quad \xymatrix{ & \ar@{-}[r]|U\ar@{-}[d] \ar@{-}@/_1pc/ @{{}{ }{}} [r]_{e}  
 & \ar@{-}[d] \\
                                                            & \ar@{-}[r]|V &} \begin{CD}\\\\ \textrm{, \quad \quad \quad with } \partial(e) U=V.\end{CD}
\end{equation}

The composition of morphisms $U\stackrel{e}{\rightarrow}V$ and 
$V\stackrel{f}{\rightarrow}W$ is defined to be $U\stackrel{fe}{\longrightarrow}W$, and the monoidal structure $\otimes$ is expressed as $U \tn V=UV$ on objects, and, on morphisms, as:
\[{\left (\quad \quad \begin{CD}
U\\@V(U,e)VV \\ V\end{CD} \,\,\right )\otimes\,\, \left( \quad \quad \begin{CD} W\\@V{(W,f)}VV\\X\end{CD}   \quad \right)= \, \begin{CD} UW \\ @VV{\big(UW,(V\tr f)\, e\big)}V\\ WX\end{CD}\,\,\,  \quad \quad  \quad \quad \,\, .}
\] 
These algebraic operations are shown using squares in  \eqref{ver} and \eqref{hor}. 
\begin{equation}\label{ver}
\xymatrix{& \ar@{-}[r]|U\ar@{-}[d] \ar@{-}@/_1pc/ @{{}{ }{}} [r]_{e} & \ar@{-}[d] \\& \ar@{-}[r]|V\ar@{-}[d] \ar@{-}@/_1pc/ @{{}{ }{}} [r]_{f} & \ar@{-}[d]\\& \ar@{-}[r]|W &
                                                            }\quad\quad \quad \xymatrix{ \quad \\ =}  {\xymatrix{ & \ar@{-}[r]|U\ar@{-}[dd] \ar@{-}@/_2pc/ @{{}{ }{}} [r]|{fe}  
 & \ar@{-}[dd] \\\\
                                                            & \ar@{-}[r]|W&}}
\end{equation}
\begin{equation}\label{hor}
 \xymatrix{ & \ar@{-}[r]|U\ar@{-}[d] \ar@{-}@/_1pc/ @{{}{ }{}} [r]_{e}  
 & \ar@{-}[d] \\
                                                            & \ar@{-}[r]|V &}
\quad \quad \quad {\begin{CD} \\ \\  \tn \end{CD} }  \xymatrix{ & \ar@{-}[r]|{U'}\ar@{-}[d] \ar@{-}@/_1pc/ @{{}{ }{}} [r]_{e'}  
 & \ar@{-}[d] \\
                                                            & \ar@{-}[r]|{V'} &} \quad \quad \quad { \begin{CD} \\ \\  = \end{CD} } 
 \xymatrix{ & \ar@{-}[rr]|{UU'}\ar@{-}[d] \ar@{-}@/_1pc/ @{{}{ }{}} [rr]_{{(V \tr e')\, e}}  &
 & \ar@{-}[d] \\
                                                            & \ar@{-}[rr]|{VV'} & &}
\end{equation}
The (strict) associativity of the composition and of the tensor product are trivial to check. The functoriality of the tensor product (also known as the interchange law) follows from the  {2nd} Peiffer equation. This calculation is done for example in {\cite{BM,BHS,Po}. }

Let $\k$ be any commutative ring. The monoidal category $\C(\Gc)$ has a $\k$-linear version $\Cl(\Gc)$, whose objects are the same as the objects of $\C(\Gc)$, but such that the set of morphisms $U \to V$ in $\Cl(\Gc)$ is given by the set of  all $\k$ linear combinations of morphisms $U \to V$ in $\C(\Gc)$. The composition and tensor product of morphisms in $\Cl(\Gc)$ are the obvious linear extensions of the  ones in $\C(\Gc)$. It is easy to see that $\Cl(\Gc)$ is a monoidal category.
The categorical group formalism is very well matched to the category of tangles to be used in the next section.

\section{Reidemeister $\Gc$ colourings of oriented tangle diagrams}
\subsection{Categories of tangles}
Tangles are a simultaneous generalization of braids and links. We follow \cite{O,K} very closely, to which we refer for more {details}. Recall that an embedding of a manifold $T$ in a manifold $M$ is said to be neat if $\partial( T)=T \cap \partial( M)$.
\begin{definition}
 An oriented tangle \cite{K,Tu,O} is a 1-dimensional smooth {oriented} manifold neatly embedded in $ \R \times \R\times [-1,1]$, such that $\partial(T) \subset \mathbb{N} \times\{0\} \times \{\pm 1\}$. A framed  oriented tangle is a tangle together with a choice of a framing in each of its components \cite{K,O}. Alternatively we can see a framed tangle as an embedding of a ribbon into $\R \times \R \times [-1,1]$, \cite{RT,Tu,CP}.
\end{definition}
\begin{definition}
 Two oriented tangles (framed oriented tangles) are said to be equivalent if they are related by an isotopy of $\R \times \R \times [0,1]$, relative to the boundary.
\end{definition}

 \begin{definition}A tangle diagram is a diagram of a tangle in $\R \times [-1,1]$, obtained from a tangle by projecting it onto $\R \times \{0\} \times [-1,1]$. Any tangle diagram unambiguously gives rise to a tangle, up to equivalence.
\end{definition}

We have  monoidal categories of oriented tangles and of framed oriented tangles \cite{K,Tu}, where composition is the obvious vertical juxtaposition of tangles and the tensor product $T \tn T'$ is obtained {by placing $T'$ on the right hand side of $T$.} The objects of the categories of oriented tangles and of framed oriented tangles are words in the symbols  $\{+,-\}$; see figure \ref{tangle} for conventions. 
\begin{figure}
\centerline{\relabelbox 
\epsfysize 2cm 
\epsfbox{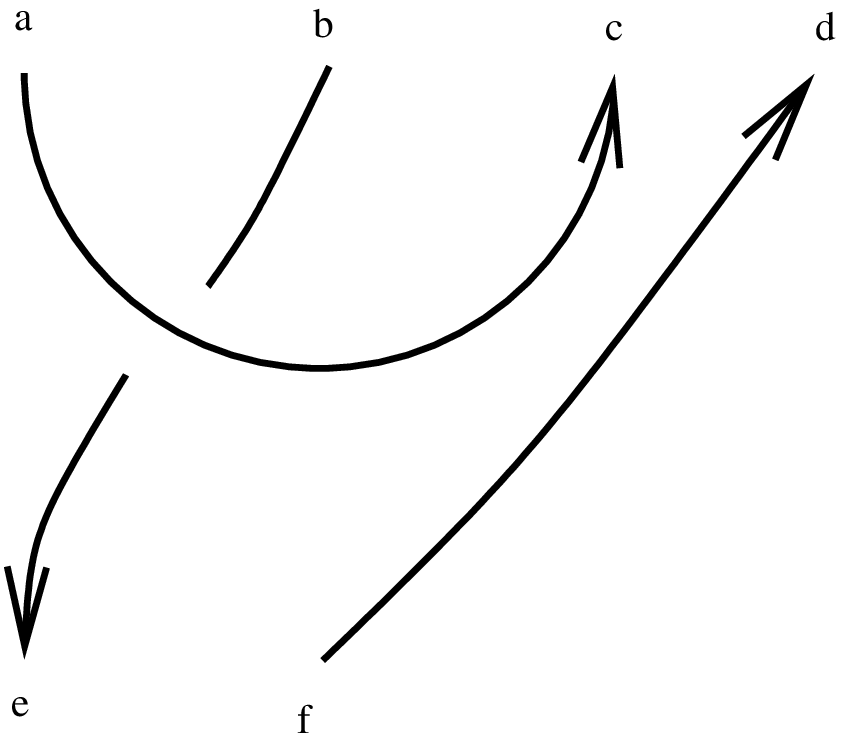}
\endrelabelbox}
\caption{A tangle with source  $++--$ and target $+-$.\label{tangle}
}
\end{figure}

{An oriented tangle diagram is a union of the {tangle diagrams} of figure \ref{dgenerators}, with some vertical lines connecting them. }
\begin{figure}
\centerline{\relabelbox 
\epsfxsize 10cm 
\epsfbox{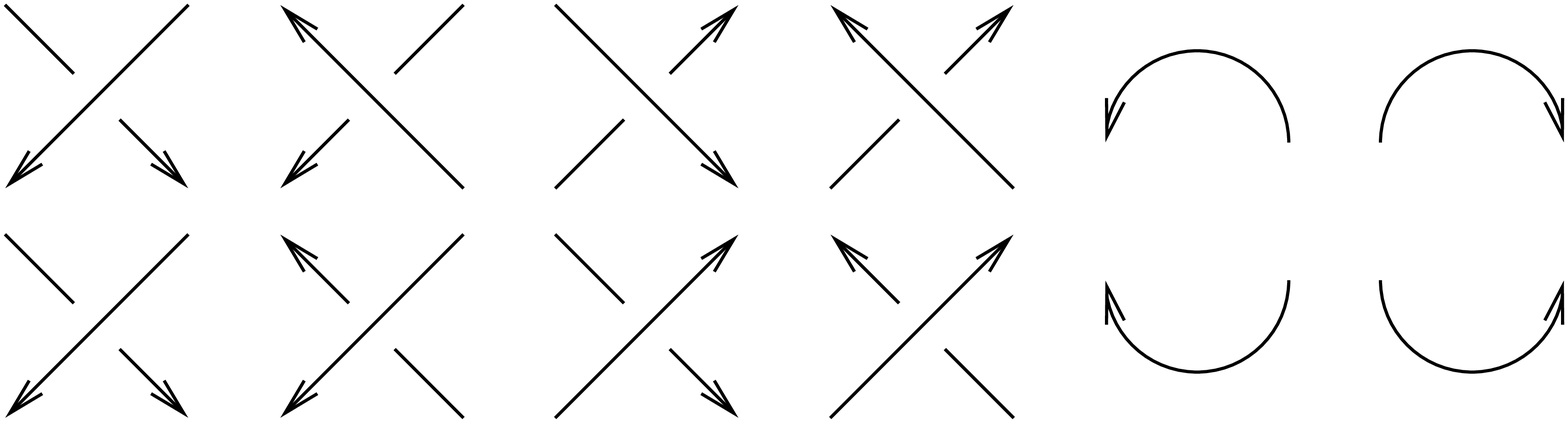}
\endrelabelbox}
\caption{Elementary generators of tangle diagrams.\label{dgenerators}}
\end{figure}
This is a redundant set if we consider oriented tangles up to isotopy.
\begin{definition}
 A sliced oriented tangle diagram is an oriented tangle diagram, subdivided into thin horizontal strips, inside which we have only vertical lines and possibly one of the morphisms in figure \ref{tangle-generators}.
\end{definition}

A theorem appearing in \cite{K}  (Theorem XII.2.2) and also in \cite{FY,RT,Tu,O} states that the category of oriented tangles {may be presented in terms of generators and relations as follows:}
\begin{Theorem}
The monoidal category of oriented tangles is equivalent to the monoidal category presented  by the six oriented tangle diagram generators :
\begin{equation}
X_+, \, X_{-},\, \cup, \, \stackrel{\leftarrow}{\cup}, \, \cap, \, \stackrel{\leftarrow}{\cap}, \,
\label{tanglegens} 
\end{equation}
shown in Figure \ref{tangle-generators}, subject to the 15 tangle diagram relations $R0A-D$, $R1$, $R2A-C$, $R3$ of Figure \ref{tangle-relations}. The category of framed oriented tangles has the same set \eqref{tanglegens} as generators, subject to the 15 relations $R0A-D$, $R1'$, $R2A-C$, $R3$ of Figure \ref{tangle-relations}. 
\end{Theorem}

\begin{remark}
We have replaced the R3 relation of Kassel's theorem with its inverse,  since this is slightly more convenient algebraically. The two forms of R3 are equivalent because of the R2A relations. The six other types of oriented crossing
in figure \ref{dgenerators}, with one or both arcs pointing upwards, can be expressed in terms of the generators (\ref{tanglegens}) and are therefore not independent generators  - see \cite{K}, Lemma XII.3.1.
\end{remark}
\begin{figure}
\centerline{\relabelbox 
\epsfysize 2cm 
\epsfbox{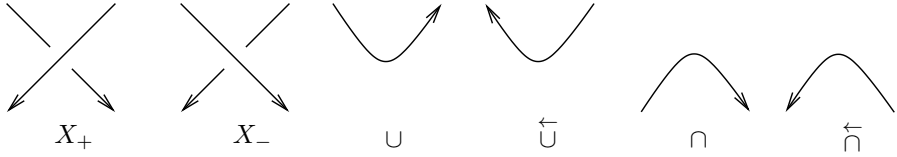}
\relabel{A}{${X_+}$}
\relabel{B}{${X_-}$}
\relabel{C}{${\cup}$}
\relabel{D}{${\stackrel{\leftarrow}{\cup}}$}
\relabel{E}{${ \cap}$}
\relabel{F}{${\stackrel{\leftarrow}{\cap}}$}
\endrelabelbox}
\caption{Generators for the categories of oriented tangles and of  framed oriented tangles.\label{tangle-generators}
}
\end{figure}
\begin{figure}
\centerline{\relabelbox 
\epsfysize 9cm 
\epsfbox{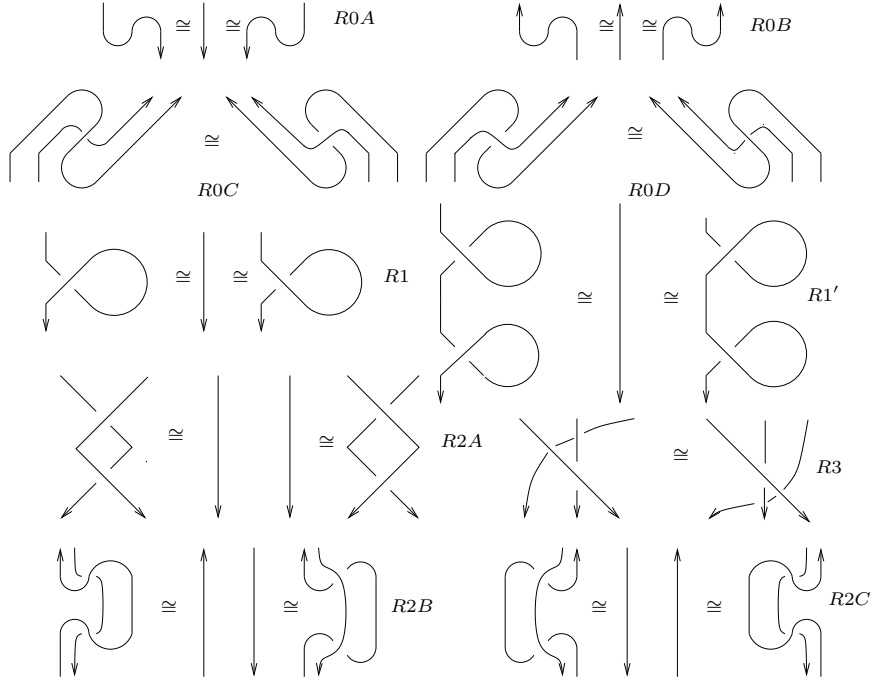}
\relabel{e1}{$\s{\cong}$}
\relabel{e2}{$\s{\cong}$}
\relabel{e3}{$\s{\cong}$}
\relabel{e4}{$\s{\cong}$}
\relabel{e5}{$\s{\cong}$}
\relabel{e6}{$\s{\cong}$}
\relabel{e7}{$\s{\cong}$}
\relabel{e8}{$\s{\cong}$}
\relabel{e9}{$\s{\cong}$}
\relabel{e10}{$\s{\cong}$}
\relabel{f}{$\s{\cong}$}
\relabel{e11}{$\s{\cong}$}
\relabel{e12}{$\s{\cong}$}
\relabel{e13}{$\s{\cong}$}
\relabel{e14}{$\s{\cong}$}
\relabel{e15}{$\s{\cong}$}
\relabel{e16}{$\s{\cong}$}
\relabel{R1}{$\s{R0A}$}
\relabel{R2}{$\s{R0B}$}
\relabel{R3}{$\s{R0C}$}
\relabel{R4}{$\s{R0D}$}
\relabel{R5}{$\s{R1}$}
\relabel{R6}{$\s{R1'}$}
\relabel{R7}{$\s{R2A}$}
\relabel{R8}{$\s{R3}$}
\relabel{R9}{$\s{R2B}$}
\relabel{R10}{$\s{R2C}$}
\endrelabelbox}
\caption{Relations for the categories of oriented tangles and of  framed oriented tangles.\label{tangle-relations}
}
\end{figure}
The previous theorem gives generators and relations at the level of tensor categories. If we want to express not-necessarily-functorial invariants of tangles it is more useful to work with sliced tangle diagrams. The following appears for example {in \cite{O}:}
\begin{Theorem}\label{sliced-tangles}
 Two sliced oriented tangle diagrams represent the same oriented tangle (framed oriented tangle) if, and only if, they are related by 
\begin{enumerate}
 \item Level preserving {isotopies} of tangle diagrams.
\item The moves $R0A-R0D$, $R1$, $R2A-R2C$, $R3$ of Figure \ref{tangle-relations} (in the case of tangles), performed locally in a diagram, or the moves $R0A-R0D$, $R1'$, $R2A-R2C$, $R3$ of Figure \ref{tangle-relations}, in the case of framed tangles.
\item The ``identity'' and ``interchange'' moves of figure \ref{tensor}. (Here $T$ and $S$ can be any tangle diagrams and a trivial tangle diagram is a diagram made only of vertical lines.)
\end{enumerate}
\end{Theorem}
\begin{figure}
\centerline{\relabelbox 
\epsfxsize 11cm 
\epsfbox{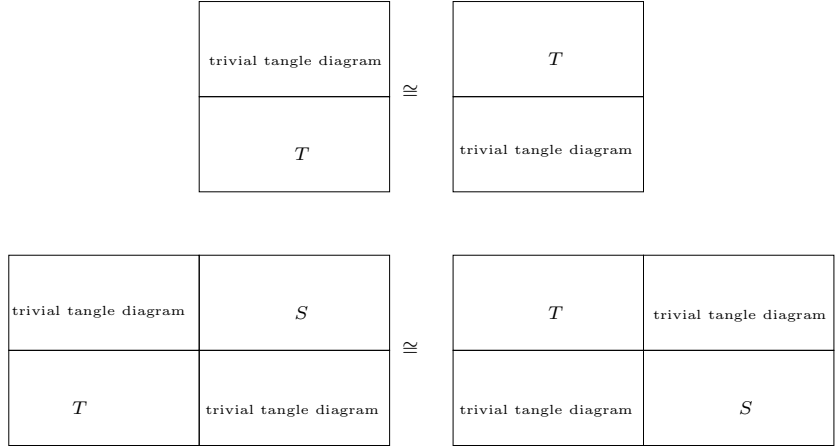}
\relabel{i1}{$\scriptscriptstyle{\textrm{trivial tangle diagram}}$}
\relabel{i2}{$\scriptscriptstyle{\textrm{trivial tangle diagram}}$}
\relabel{i3}{$\scriptscriptstyle{\textrm{trivial tangle diagram}}$}
\relabel{i4}{$\scriptscriptstyle{\textrm{trivial tangle diagram}}$}
\relabel{i5}{$\scriptscriptstyle{\textrm{trivial tangle diagram}}$}
\relabel{i6}{$\scriptscriptstyle{\textrm{trivial tangle diagram}}$}
\relabel{A}{$\s{T}$}
\relabel{B}{$\s{T}$}
\relabel{T1}{$\s{T}$}
\relabel{S1}{$\s{S}$}
\relabel{T2}{$\s{T}$}
\relabel{S2}{$\s{S}$}
\relabel{s}{$\s{\cong}$}
\relabel{t}{$\s{\cong}$}
\endrelabelbox}
\caption{The identity move and the interchange move.\label{tensor}}
\end{figure}

\begin{definition}
[Enhanced tangle]\label{ent} Let $X$ be a set (normally $X$ will be either a group or a {quandle / rack}). An $X$-enhanced (framed) oriented tangle is a (framed) oriented tangle  $T$ together with an assignment of an element of $X$ to each point of the boundary of $T$. We will  consider $X$-enhanced (framed)  tangles up to isotopy of $\R \times \R \times [-1,1]$, fixing the end-points. 
\end{definition}

Given a set $X$, there exist monoidal categories having as morphisms the set of $X$-enhanced oriented tangles and of $X$-enhanced framed oriented tangles, up to isotopy.  These categories have as objects the set of all formal words $\w$ in the symbols $a$ and $a^*$ where $a \in X$.  See figure \ref{etangle} for conventions.
\begin{figure}
\centerline{\relabelbox 
\epsfysize 2cm 
\epsfbox{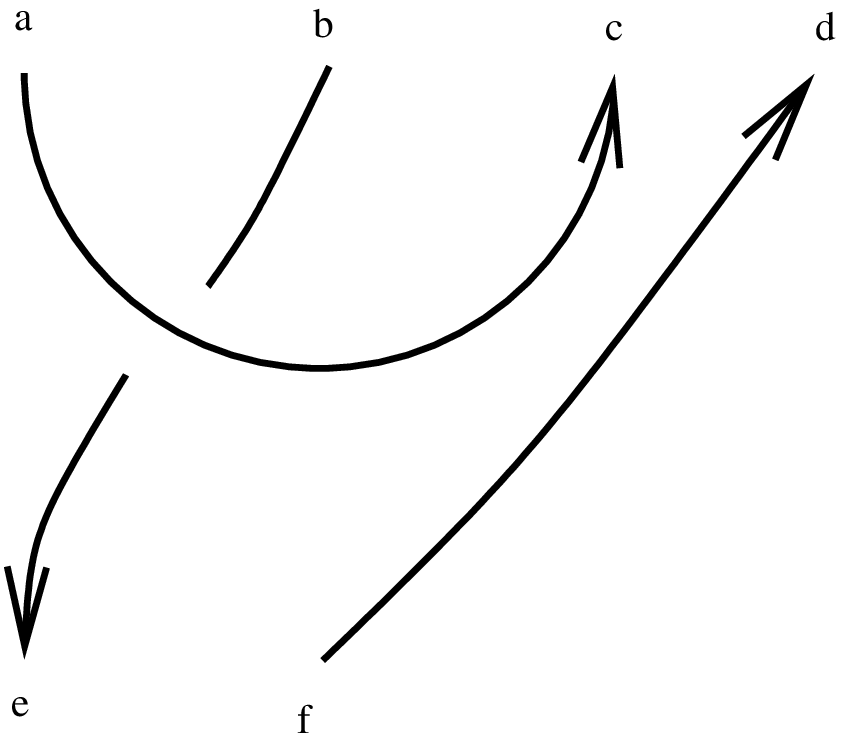}
\relabel{a}{$\s{a}$}
\relabel{b}{$\s{b}$}
\relabel{c}{$\s{c}$}
\relabel{d}{$\s{d}$}
\relabel{e}{$\s{e}$}
\relabel{f}{$\s{f}$}
\endrelabelbox}
\caption{An $X$-enhanced tangle with source $a.b.c^*.d^*$ and target $e.f^*$.\label{etangle}}
\end{figure}
\begin{definition}\label{ev1}
Let $G$ be a group. There exists an evaluation map $\w \mapsto e(\w)$, which associates to a word $\w$ in $G \sqcup G^*$ an element of $G$, obtained by multiplying all elements of $\w$ in the same order, by putting $g^*\doteq g^{-1}$ (and $e(\emptyset)=1_G$ for the empty word {$\emptyset$}).
\end{definition}
 \subsection{Colourings of tangle diagrams} Let ${\Gc= ( \d\colon E \to  G,\tr)}$ be a crossed module of groups.
We wish to define the notion of a $\Gc$-colouring of an oriented tangle diagram, by assigning elements of $G$ to the arcs and elements of $E$ to the crossings in a suitable way. 

For a link diagram $D$ realized as a tangle diagram, a 
$\Gc$-colouring may be regarded as a morphism of crossed modules from the fundamental crossed module $\Pi_2(X_D,Y_D)$ of {Example \ref{fundexp}}, to $\Gc$. This idea extends in a natural way to general tangle diagrams. 

\begin{definition}
Given a finite crossed module ${\Gc= ( \d\colon E \to  G,\tr)}$ and an oriented tangle diagram $D$, a $\Gc$-colouring of $D$ is an assignment of an element of $G$ to each arc of the diagram, and of an element of $E$ to each crossing of the diagram, such that, at each crossing of type $X_+$ or $X_{-}$ with colourings as in \eqref{GC-col}, the following relations hold:
\begin{eqnarray}
X_+: \quad \d(e) & = & XYX^{-1}Z^{-1} \label{fundX+}\\
X_-: \quad \d(e) & = & YXZ^{-1}X^{-1} \label{fundX-}
\end{eqnarray}
%\end{definition}
\begin{equation}\label{GC-col}
 \xymatrix{ &Z\ar[dr] & &X \ar[ddll]  \\
            & &{\quad\quad e}\ar[dr] &\\
            &X & &Y}  \xymatrix{ &X\ar[ddrr] & &Z \ar[dl]  \\
            & &{\quad\quad e}\ar[dl] &\\
            &Y & &X}
\end{equation}
\end{definition}

Thus we are assigning to each type of coloured crossing a morphism of $\mathcal{C}(\Gc)$ and of $\Cl(\Gc)$, and in a similar way we may associate morphisms of $\mathcal{C}(\Gc)$ to all elementary $\Gc$-coloured tangles, as summarised in figure
\ref{tanglemor}. With the duality where the dual of the morphism $X \ra{e} \partial(e) X$ is $X^{-1} \partial(e)^{-1} \ra{X^{-1} \tr e}  X^{-1}$, and the morphisms associated to the cups and caps are the ones in figure \ref{tanglemor}, we can easily see \cite{FM} that the  monoidal category $\C(\Gc)$ is a compact category \cite{RT}, and in fact a pivotal category \cite{BW}, which however is not spherical in general. Therefore planar graphs coloured in $\C(\Gc)$ can be evaluated to give morphisms in $\C(\Gc)$, and this evaluation is invariant under planar isotopy.

Thus we can assign to the complete $\Gc$-coloured oriented tangle diagram a morphism of $\mathcal{C}(\Gc)$, by using the monoidal product horizontally and composition vertically. This leads to the following definition:
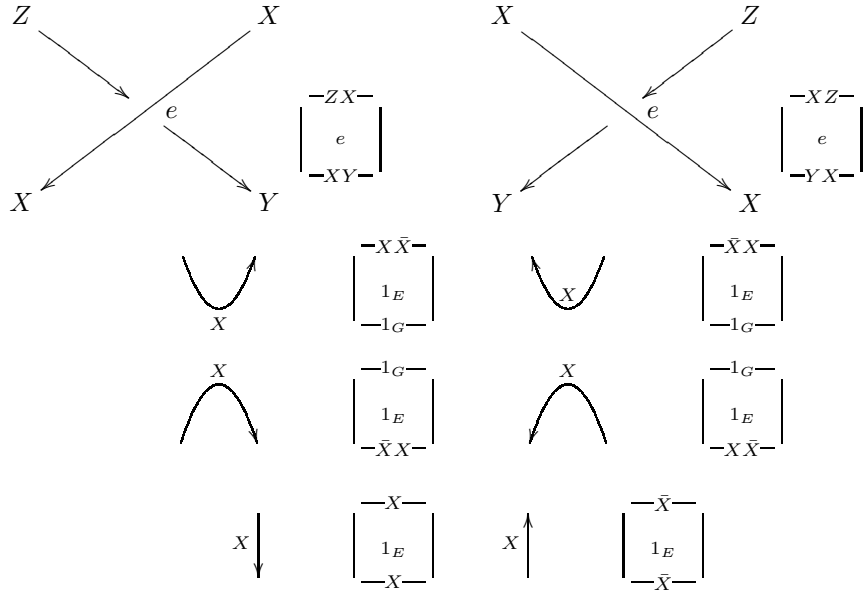
\begin{figure}
 \[\hskip-1.2cm {\xymatrix{ &Z\ar[dr] & &X \ar[ddll]  \\
            & &{\quad\quad e}\ar[dr] &\\
            &X & &Y} \hskip-1cm   \xymatrix{\\ & \ar@{-}[r]|{ZX}\ar@{-}[d] \ar@{-}@/_1pc/ @{{}{ }{}} [r]_{e}  
 & \ar@{-}[d] \\
                                                            & \ar@{-}[r]|{XY} &} \hskip0.2cm \xymatrix{ &X\ar[ddrr] & &Z \ar[dl]  \\
            & &{\quad\quad e}\ar[dl] &\\
            &Y & &X} \hskip-1cm  \xymatrix{ \\ & \ar@{-}[r]|{XZ}\ar@{-}[d] \ar@{-}@/_1pc/ @{{}{ }{}} [r]_{e}  
 & \ar@{-}[d] \\
                                                      & \ar@{-}[r]|{YX} &}}\]
\[
\xymatrix{ & \ar@/_2pc/[r]_X & } \xymatrix{  & \ar@{-}[r]|{X\bar{X}}\ar@{-}[d] \ar@{-}@/_1pc/ @{{}{ }{}} [r]_{1_E}  
 & \ar@{-}[d]\\
                                                      & \ar@{-}[r]|{1_G} &} \xymatrix{  & & \ar@/^2pc/[l]_X  } \xymatrix{ 		 & \ar@{-}[r]|{\bar{X}X}\ar@{-}[d] \ar@{-}@/_1pc/ @{{}{ }{}} [r]_{1_E}  
 & \ar@{-}[d]\\
                                                      & \ar@{-}[r]|{1_G} &} \]
\[
\xymatrix{ 		\\& \ar@/^2pc/[r]^X & } \xymatrix{  & \ar@{-}[r]|{1_G}\ar@{-}[d] \ar@{-}@/_1pc/ @{{}{ }{}} [r]_{1_E}  
 & \ar@{-}[d]\\
                                                      & \ar@{-}[r]|{\bar X X} &} \xymatrix{ \\ & & \ar@/_2pc/[l]_X  } \xymatrix{ 		 & \ar@{-}[r]|{1_G}\ar@{-}[d] \ar@{-}@/_1pc/ @{{}{ }{}} [r]_{1_E}  
 & \ar@{-}[d]\\
                                                      & \ar@{-}[r]|{X \bar X} &} \]

\[
\xymatrix{& \ar[d]_X \\& } \xymatrix{  & \ar@{-}[r]|{X}\ar@{-}[d] \ar@{-}@/_1pc/ @{{}{ }{}} [r]_{1_E}  
 & \ar@{-}[d]\\
                                                      & \ar@{-}[r]|{ X} &} \xymatrix{  & \\ & \ar[u]^X  } \xymatrix{ 		 & \ar@{-}[r]|{\bar X}\ar@{-}[d] \ar@{-}@/_1pc/ @{{}{ }{}} [r]_{1_E}  
 & \ar@{-}[d]\\
                                                      & \ar@{-}[r]|{\bar{X}} &} \]
\caption{Turning $\Gc$-coloured tangles into morphisms of $\C(\Gc)$.  The symbol {$\bar{X}$ stands for $X^{-1}$.}\label{tanglemor} }
\end{figure}

\begin{definition}\label{ev2}
Given a $\Gc$-colouring $F$ of a tangle diagram $D$ (we say $F\in C_\Gc(D)$, the set of $\Gc$ colourings of $D$), the evaluation of $F$, denoted $e(F)$, is the morphism in $\mathcal{C}(\Gc)$ 
 obtained by multiplying horizontally and composing vertically the morphisms of $\mathcal{C}(\Gc)$ associated to the elementary tangles which make up $D$. 
\end{definition}

\begin{remark}
For a link diagram the evaluation of $F$ takes values in $A$, the automorphism subgroup of $1_G$, i.e. $A={\rm ker}\, \d\subset E$. 
\end{remark}

For a tangle diagram without open ends at the top and bottom, i.e. a link diagram, one can conjecture that the number of  $\Gc$-colourings of the diagram can be normalized to a link invariant, by analogy with the familiar link invariant which is the number of Wirtinger colourings of the diagram using a finite group $G$ (a Wirtinger colouring in the present context would be a $\Gc$-colouring where the group $E$ is trivial).   {Indeed it was proven in \cite{FM} that the number of colourings of a link diagram evaluating to the identity of $E$ can be normalised to give an invariant of knots}. However this invariant depends only on the homotopy type of the complement of the knot \cite{FM2}, thus it is a function of the knot group only.

Therefore we are led to consider the possibility of imposing more refined constraints on the $\Gc$-colourings of a tangle diagram in such a way that the number of constrained $\Gc$-colourings does respect the Reidemeister moves. Intuitively we are looking at the simple homotopy type, rather than the homotopy type of {a link complement.} Our idea is to restrict ourselves to $\Gc$-colourings of diagrams where, at each crossing, the colouring of the crossing with an element of $E$ is determined by the $G$-colouring of two arcs, namely the overcrossing arc and the lower undercrossing arc. To this end we introduce two functions:
\[
\psi: G\times G \rightarrow E, \qquad \phi: G\times G \rightarrow E,
\]
which determine the $E$-colouring of the two types of crossing, as in \eqref{psiphidef}:
\begin{equation}\label{psiphidef}
 \xymatrix{ &Z\ar[dr] \ar@{-}@/_2pc/ @{{}{ }{}}[rr]_{\quad \quad \quad \quad \quad \psi(X,Y)} & &X \ar[ddll]  \\
            & &\ar[dr] &\\
            &X & &Y}  \xymatrix{ &X\ar[ddrr] \ar@/_2pc/ @{{}{ }{}}[rr]_{\quad \quad \quad \quad \quad \phi(X,Y)}  & &Z \ar[dl]  \\
            & & \ar[dl] &\\
            &Y & &X}
\end{equation}
{Since this is  a $\Gc$-colouring, these functions  determine the $G$-assignment for the remaining arc:}
\begin{eqnarray}
X_+: \quad Z & = & \d \psi(X,Y)^{-1}XYX^{-1} \label{Zforpsi} \\
X_-: \quad Z & = & X^{-1}\d  \phi(X,Y)^{-1}YX \label{Zforphi}
\end{eqnarray}

We now come to our main definition.
\begin{definition}[unframed Reidemeister pair]\label{rp}
 The pair  $\Phi=(\psi,\phi)$ is said to be an unframed Reidemeister pair if $\psi \colon G \times G \to E$ and $\phi\colon G \times G \to E$ satisfy the following three relations for each $ X,\, Y, \, T \in G$:
\begin{eqnarray}
 \psi(X,X) & \stackrel{R1}{=} & 1_E \label{R1}\\
 \phi(X,Y) \psi(X,Z) & \stackrel{R2}{=} & 1_E \label{R2}\\				
\f(Y,X). Y\tr \f(T,Z).\f(T,Y) & \stackrel{R3}{=} & X\tr \f(T,Y). \f (T,X) . T\tr \f(V,W) \label{R3}
\end{eqnarray}
where, in R2,  
$ Z = X^{-1}\d  \phi(X,Y)^{-1}YX $,
and in R3:
\begin{align*}
 Z & =  Y^{-1}\d  \phi(Y,X)^{-1}XY \,\,, 
&V & =  T^{-1}\d  \phi(T,Y)^{-1}YT\,\,,   
&W & =  T^{-1}\d  \phi(T,X)^{-1}XT 
\end{align*}
\end{definition}
\begin{remark}
The equations above relate to the Reidemeister 1-3 moves, as we will see shortly in the proof of Theorem \ref{maintheorem}.
 If equation \eqref{R2} holds we can substitute \eqref{R3} by the equivalent:
 \begin{equation}\label{R3p} 
 \psi(X,Y)\,.\, A \tr \psi(X,Z) \,.\, \psi(A,B)= X \tr \psi(Y,Z) \,.\, \psi(X,C)\,.\, D \tr \psi(X,Y)
\end{equation}
where:
\begin{align*}
  A&=\d(\psi(X,Y))^{-1} XYX^{-1}\,\,,
  &B&=\d(\psi(X,Z))^{-1} XZX^{-1}\,\,,\\
  C&=\d(\psi(Y,Z))^{-1} YZY^{-1}\,\,,
  &D&=\d(\psi(X,C))^{-1} XCX^{-1}\,\,.
\end{align*}
\end{remark}

\begin{figure}
\centerline{\relabelbox 
\epsfysize 3cm 
\epsfbox{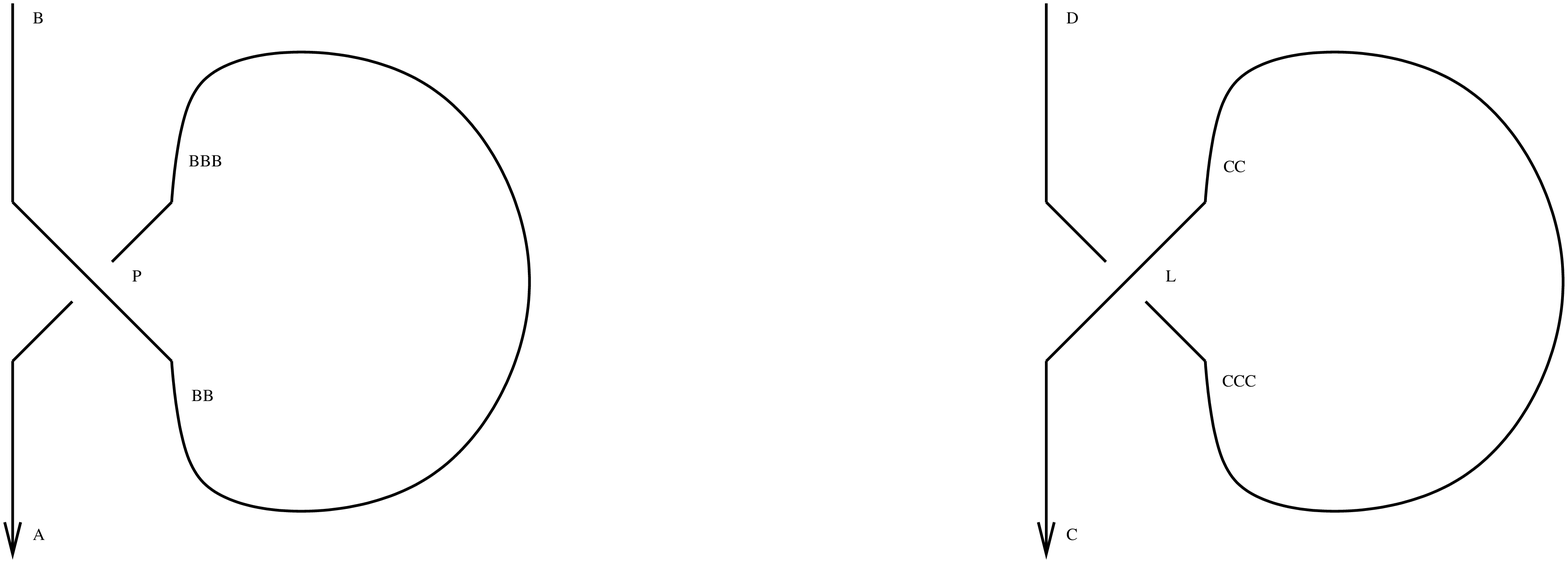}
\relabel{A}{$\scriptstyle{Z}$}
\relabel{B}{$\scriptstyle{f(Z)}$}
\relabel{BB}{$\scriptstyle{f(Z)}$}
\relabel{BBB}{$\scriptstyle{f(Z)}$}
\relabel{P}{$\scriptstyle{\phi(f(Z),Z)}$}
\relabel{C}{$\scriptstyle{A}$}
\relabel{CC}{$\scriptstyle{A}$}
\relabel{CCC}{$\scriptstyle{A}$}
\relabel{D}{$\scriptstyle{g(A)}$}
\relabel{L}{$\scriptstyle{\psi(A,A)}$}
\endrelabelbox}
\caption{Definition of $f,g \colon G \to G$.}
\label{fandg}
\end{figure}

\begin{definition}[Framed Reidemeister pair]\label{frp}
 The pair $\Phi=(\psi,\phi)$ is said to be a framed Reidemeister pair if relations $R2$ and $R3$ of Definition \ref{rp} hold and moreover:
\begin{enumerate}
 \item Given $Z$ in $G$, the equation (c.f. left of figure \ref{fandg}) {$\partial(\phi(A,Z))A=Z$} has a unique solution  $f(Z) \in G$.
\item Defining $g(A)=\d(\psi(A,A))^{-1}A$ (c.f. right of figure \ref{fandg}) it holds that $f\circ g=g\circ f=\id$. In particular both $f$ and $g$ are bijective.
\end{enumerate}
\end{definition}

\begin{definition}
Given a crossed module ${\Gc= ( \d\colon E \to  G,\tr)}$, with $G$ finite, provided with a (framed or unframed) Reidemeister pair $\Phi=(\psi,\phi)$, and an oriented $G$-enhanced tangle diagram $D$, a Reidemeister $\Gc$-colouring of $D$ is a $\Gc$-colouring of $D$ (which extends the colourings at the end-points of $D$, in the sense that an arc coloured by $g$ corresponds to endpoints coloured by $g$ or $g^\ast$, depending on the orientation - see Figure \ref{etangle}) determined by the functions $\psi: G\times G \rightarrow E,\,  \phi: G\times G \rightarrow E$, which fix the colourings at each crossing as in \eqref{psiphidef}, \eqref{Zforpsi} and \eqref{Zforphi}.
\end{definition}

We are now in a position to define a state-sum coming from the Reidemeister $\Gc$-colourings of a link diagram $D$. {Recall Definitions \ref{ev1} and \ref{ev2}.}

\begin{definition}
Consider a  crossed module ${\Gc= ( \d\colon E \to  G,\tr)}$, with $G$ finite, provided with a (framed or unframed) Reidemeister pair $\Phi=(\p,\f)$. Consider an oriented {$G$-enhanced} tangle diagram $D$, connecting the {words $\w$ and $\w'$ in $G \sqcup G^*$}.
We denote the corresponding set of Reidemeister  $\Gc$-colourings of $D$ by $C_\Phi(D,\w,\w')$. Then we define the state sum:
\begin{equation}
I_\Phi(D) =\langle \w|I_\Phi(D)| \w' \rangle\doteq\sum_{F\in C_\Phi(D,\w,\w')} e(F)
\label{Iphi}
\end{equation}
taking values in ${\mathbb N}\big[{\rm Hom}_{\C(\Gc)}\big(e(\w),e(\w')\big)\big]\subset {\rm Hom}_{\C_{\Z}(\Gc)}\big(e(\w),e(\w')\big).$
(Here the set of morphisms $x \to y$ in a category $\C$ is denoted by ${\rm Hom}_\C(x,y)$.)
\end{definition}
\begin{remark} If $D$ is a link diagram then $I_\Phi(D)$ takes values in $\Z[A]$,  the group algebra of $A = \ker \partial.$
\end{remark}
\begin{Theorem}
The state sum $I_\Phi$ defines an invariant of $G$-enhanced tangles if $\Phi$ is an unframed Reidemeister pair and an invariant of framed $G$-enhanced tangles if $\Phi$ is a framed Reidemeister pair.
\label{maintheorem}
\end{Theorem}
\begin{proof}
{To prove this result, we} need to show that $I_\Phi$ respects the relations of Theorem \ref{sliced-tangles}. Invariance under level preserving isotopy is obvious. Let us now address, for the unframed case, the moves $R0A-D$, $R1$, $R2A-C$, $R3$ of Figure \ref{tangle-relations}. For each relation we fix matching colours on the maximum number of arcs connecting to the exterior, and then show that the corresponding morphisms of $\mathcal{C}(\Gc)$ are equal (and in some cases, that the remaining arcs connecting to the exterior are also coloured compatibly). Thus for a pair of diagrams related by one of the relations, each term in the expression for $I_\Phi$ for one diagram has a corresponding term in the expression for $I_\Phi$ for the other diagram, and the evaluations are equal term by term. 
\begin{figure}
\centerline{\relabelbox 
\epsfysize 9cm 
\epsfbox{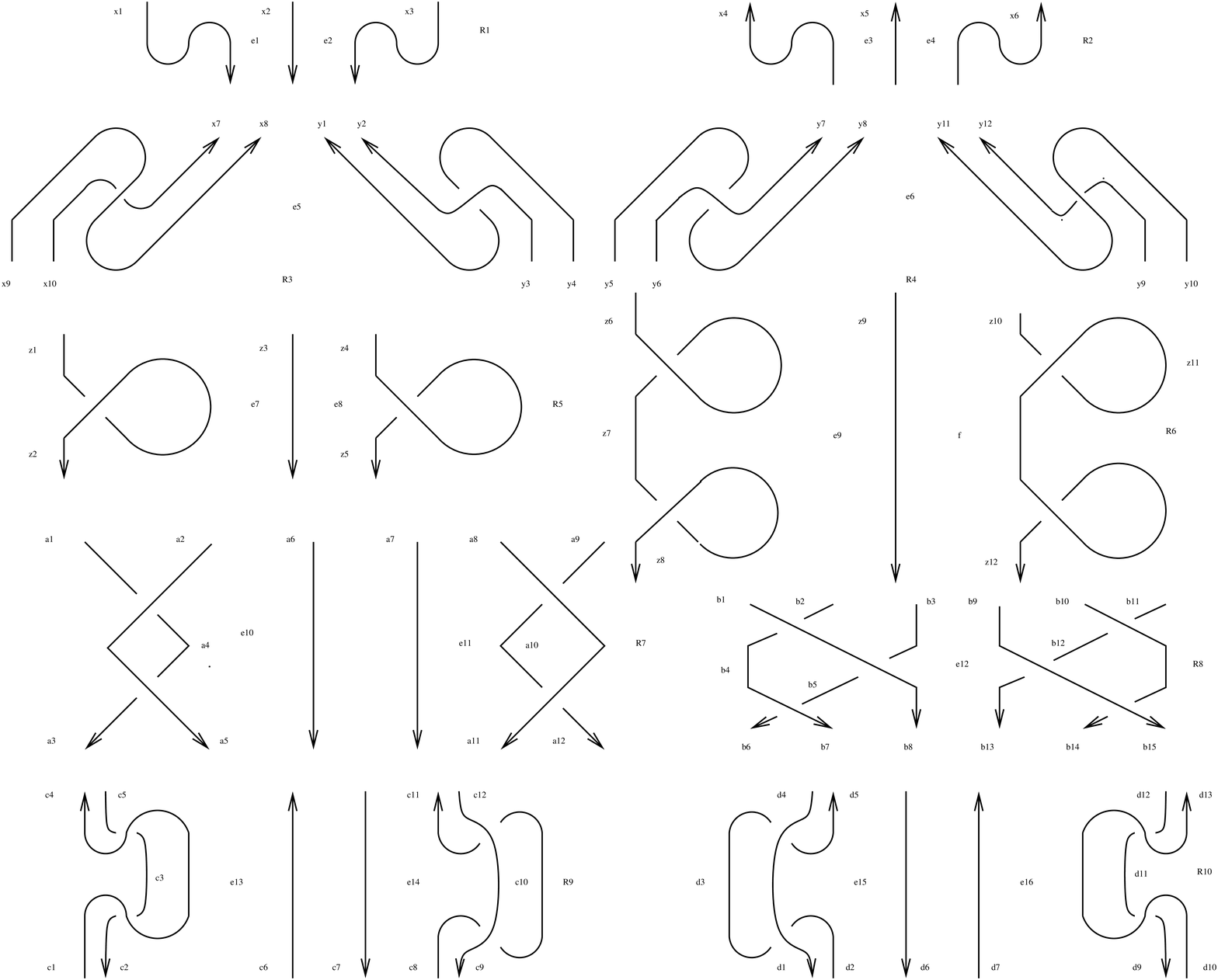}
\relabel{e1}{$\s{\cong}$}
\relabel{e2}{$\s{\cong}$}
\relabel{e3}{$\s{\cong}$}
\relabel{e4}{$\s{\cong}$}
\relabel{e5}{$\s{\cong}$}
\relabel{e6}{$\s{\cong}$}
\relabel{e7}{$\s{\cong}$}
\relabel{e8}{$\s{\cong}$}
\relabel{e9}{$\s{\cong}$}
\relabel{e10}{$\s{\cong}$}
\relabel{f}{$\s{\cong}$}
\relabel{e11}{$\s{\cong}$}
\relabel{e12}{$\s{\cong}$}
\relabel{e13}{$\s{\cong}$}
\relabel{e14}{$\s{\cong}$}
\relabel{e15}{$\s{\cong}$}
\relabel{e16}{$\s{\cong}$}
\relabel{R1}{$\s{R0A}$}
\relabel{R2}{$\s{R0B}$}
\relabel{R3}{$\s{R0C}$}
\relabel{R4}{$\s{R0D}$}
\relabel{R5}{$\s{R1}$}
\relabel{R6}{$\s{R1'}$}
\relabel{R7}{$\s{R2A}$}
\relabel{R8}{$\s{R3}$}
\relabel{R9}{$\s{R2B}$}
\relabel{R10}{$\s{R2C}$}
\relabel{x1}{$\s{X}$}
\relabel{x2}{$\s{X}$}
\relabel{x3}{$\s{X}$}
\relabel{x4}{$\s{X}$}
\relabel{x5}{$\s{X}$}
\relabel{x6}{$\s{X}$}
\relabel{x7}{$\s{Y}$}
\relabel{x8}{$\s{X}$}
\relabel{x9}{$\s{X}$}
\relabel{x10}{$\s{Z}$}
\relabel{y1}{$\s{Y}$}
\relabel{y2}{$\s{X}$}
\relabel{y3}{$\s{X}$}
\relabel{y4}{$\s{Z}$}
\relabel{y5}{$\s{Z}$}
\relabel{y6}{$\s{X}$}
\relabel{y7}{$\s{X}$}
\relabel{y8}{$\s{Y}$}
\relabel{y9}{$\s{Z}$}
\relabel{y10}{$\s{X}$}
\relabel{y11}{$\s{X}$}
\relabel{y12}{$\s{Y}$}
\relabel{z1}{$\s{W}$}
\relabel{z2}{$\s{X}$}
\relabel{z3}{$\s{X}$}
\relabel{z4}{$\s{X}$}
\relabel{z5}{$\s{Y}$}
\relabel{z6}{$\s{X}$}
\relabel{z7}{$\s{Y}$}
\relabel{z8}{$\s{Z}$}
\relabel{z9}{$\s{Z}$}
\relabel{z10}{$\s{W}$}
\relabel{z11}{$\s{V}$}
\relabel{z12}{$\s{Z}$}
\relabel{a1}{$\s{W}$}
\relabel{a2}{$\s{X}$}
\relabel{a3}{$\s{Y}$}
\relabel{a4}{$\s{Z}$}
\relabel{a5}{$\s{X}$}
\relabel{a6}{$\s{Y}$}
\relabel{a7}{$\s{X}$}
\relabel{a8}{$\s{Y}$}
\relabel{a9}{$\s{U}$}
\relabel{a10}{$\s{T}$}
\relabel{a11}{$\s{Y}$}
\relabel{a12}{$\s{X}$}
\relabel{b1}{$\s{T}$}
\relabel{b2}{$\s{V}$}
\relabel{b3}{$\s{U}$}
\relabel{b4}{$\s{Y}$}
\relabel{b5}{$\s{Z}$}
\relabel{b6}{$\s{X}$}
\relabel{b7}{$\s{Y}$}
\relabel{b8}{$\s{T}$}
\relabel{b9}{$\s{T}$}
\relabel{b10}{$\s{V}$}
\relabel{b11}{$\s{Z}$}
\relabel{b12}{$\s{W}$}
\relabel{b13}{$\s{X}$}
\relabel{b14}{$\s{Y}$}
\relabel{b15}{$\s{T}$}
\relabel{c1}{$\s{X}$}
\relabel{c2}{$\s{Y}$}
\relabel{c3}{$\s{Z}$}
\relabel{c4}{$\s{X}$}
\relabel{c5}{$\s{W}$}
\relabel{c6}{$\s{X}$}
\relabel{c7}{$\s{Y}$}
\relabel{c8}{$\s{U}$}
\relabel{c9}{$\s{Y}$}
\relabel{c10}{$\s{T}$}
\relabel{c11}{$\s{X}$}
\relabel{c12}{$\s{Y}$}
\relabel{d1}{$\s{X}$}
\relabel{d2}{$\s{W}$}
\relabel{d3}{$\s{Y}$}
\relabel{d4}{$\s{X}$}
\relabel{d5}{$\s{Z}$}
\relabel{d6}{$\s{X}$}
\relabel{d7}{$\s{Z}$}
\relabel{d9}{$\s{X}$}
\relabel{d10}{$\s{Z}$}
\relabel{d11}{$\s{T}$}
\relabel{d12}{$\s{U}$}
\relabel{d13}{$\s{Z}$}
\endrelabelbox}
\caption{Tangle relations with $G$-assignments to the arcs.\label{tangle-relations-col}
}
\end{figure}

%\vskip 0.3cm
 R0A and R0B: Fix $X\in G$. The corresponding equation in $E$ is $1_E=1_E$ in each case.

%\vskip 0.3cm
 R0C: Fix $X,Y\in G$. The corresponding equation in $E$:
\[
(X^{-1}Z^{-1})\tr \psi(X,Y) = (Y^{-1} X^{-1}) \tr \psi(X,Y)
\]
is an identity which follows from:
\[
\psi(X,Y) = (\d \psi(X,Y))\tr \psi(X,Y) = (XYX^{-1}Z^{-1}) \tr \psi(X,Y).
\]
%\vskip 0.3cm \noindent

R0D: Fix $X,Y\in G$. The corresponding equation in $E$:
\[
(Z^{-1}X^{-1})\tr \phi(X,Y) = (X^{-1} Y^{-1}) \tr \phi(X,Y)
\]
is an identity which follows from:
\[
\phi(X,Y) = (\d \phi(X,Y))\tr \phi(X,Y) = (YXZ^{-1}X^{-1}) \tr \phi(X,Y).
\]

%\vskip 0.3cm \noindent
 R2A: Fix $X,Y\in G$. The corresponding equation in $E$ is:
\begin{equation}
\phi(X,Y) \psi(X,Z)= 1_E = \psi(Y,X)\phi(Y,T) 
\label{R2v2}
\end{equation}
where $Z=X^{-1}\d \phi(X,Y)^{-1}YX$ and $T=\d \psi(Y,X)^{-1}YXY^{-1}$.
The {1st} equality is (\ref{R2}), and the {2nd} equality follows from the {1st}:
$\psi(X,Z)\phi(X,Y)=1_E$ with $Z=X^{-1} \d \phi(Y,X)^{-1}XY$, i.e. 
\[Y= \d \phi (X,Y)XZX^{-1}= 
\d \psi(X,Z)^{-1} XZX^{-1},\] then substitute variables $X\mapsto Y, \, Z\mapsto X, \, Y\mapsto T$. Applying $\d$ to (\ref{R2v2}), it follows that $W=Y$ and $U=X$.

%\vskip 0.3cm \noindent 
R1: Fix $X\in G$. The corresponding equation in $E$ is:
\[
\psi(X,X)=1_E = \phi(X,Y), \quad \d \phi(X,Y) =YX^{-1}
\]
The {1st} equality is (\ref{R1}), implying $W=X$, and the {2nd} equality follows from (\ref{R1}) and (\ref{R2v2}), {see just below, which {also} implies $Y=X$.}
\[
\phi(X,Y)=\phi(X,Y)\psi(X,X)\phi(X,X)=\phi(X,X)=1_E.
\]

%\vskip 0.3cm \noindent 

R2B: Fix $X,Y \in G$. The corresponding equation in $E$ for {the left move is:}
\[
X^{-1}\tr \phi(X,Y)\, . \, X^{-1} \tr \psi(X,Z) = 1_E
\]
with $Z=X^{-1}\d \phi(X,Y)^{-1}YX$, which is the {1st} equality in (\ref{R2v2}). 
Applying $\d$, it follows that $W=Y$.
The equation in $E$ for the move on the right:
\[
U^{-1} \tr \psi(Y,T) \, . \, X^{-1} \tr \phi(Y,X) = 1_E
\]
with $T=Y^{-1}\d \phi(Y,X)^{-1}XY$ and $U=\d \psi(Y,T)^{-1}YTY^{-1}$, { follows from the {2nd} equality of (\ref{R2v2})
for $X=\d \psi(Y,T)^{-1}YTY^{-1}$,} and therefore  $T=Y^{-1} \d \psi(Y,T)XY$, or $T = Y^{-1}\d \phi(Y,X)^{-1}XY$, by acting with 
$X^{-1}=U^{-1}$.

%\vskip 0.3cm \noindent

 R2C: Fix $X,Z \in G$. The corresponding equation in $E$:
\begin{equation}
Y^{-1}\tr\phi(X,Y) \, . \, Y^{-1}\tr\psi(X,Z) = 1_E = Z^{-1}\tr \psi(Z,X) \, . \, Z^{-1}\tr \phi(Z,T)
\label{r2c}
\end{equation}
with $Y=\d\psi(X,Z)^{-1}XZX^{-1}$ and $T=\d\psi(Z,X)^{-1}ZXZ^{-1}$, is equivalent to (\ref{R2v2}) with $Y$ substituted by $Z$ in the {2nd} equality.  Applying $\d$ to (\ref{r2c}), it follows that $W=Z$ and $U=X$.

%\vskip 0.3cm \noindent

 R3: Fix $X,Y,T \in G$. The corresponding equation in $E$ is the Reidemeister 3 equation (\ref{R3}), which also implies the equality $U=Z$, by applying $\d$. 

%\vskip 0.3cm \noindent 
Invariance under the identity and interchange moves of Figure \ref{tensor} is immediate - to show the latter we use the interchange law for the operations of Figures \ref{ver} and \ref{hor}.

%\vskip 0.3cm \noindent
 For framed tangles we need to show invariance of $I_\Phi$ under the $R1'$ move using the properties (i) and (ii) of Definition \ref{frp}, which replace the Reidemeister 1 condition (\ref{R1}).

%\vskip 0.3cm \noindent 
Fix $Z\in G$. For the move on the left, at the lower crossing we have, from (i), the relation $Y=g(Z)$, and at the upper crossing we have the relation $X=f(g(Z))=Z$, by (ii).
The equation for the move, {which is
$
\psi(Z,Z)\, \phi(Z,Y)=1_E,
$}
follows from the {2nd} equality in (\ref{R2v2}).
For the move on the right, at the lower crossing we have, from (i), the relation $V=f(Z)$, and at the upper crossing we have the relation $W=g(f(Z))=Z$, by (ii).
The equation for the move, {namely
$
\phi(V,Z)\, \psi(V,V)=1_E
$}
follows from the {1st} {equality in (\ref{R2v2}).  }
%{ We have thus finished the proof of Theorem \ref{maintheorem}.}
\end{proof}

\noindent {We close this section by stating a TQFT property of the invariant $I_\Phi$, which follows easily from the definition.} 
\begin{Theorem}
Let $D_1$ and $D_2$ be tangle diagrams, so that the vertical composition $\begin{array}{|l|}\hline D_1 \\ \hline D_2\\ \hline\end{array}$ is well defined. {For any enhancements $\w$ and $\w''$ of the top of $D_1$ and the bottom of $D_2$ we have:}
\[\left \langle \w \left | I_\Phi\left(\,\,\begin{array}{|l|}\hline D_1 \\ \hline D_2\\ \hline\end{array}\,\,\right) \right |\w'' \right \rangle=\sum_{\w'} \begin{array}{|l|}\hline \langle \w|I_\Phi(D_1)|\w' \rangle\\\hline \langle \w'|I_\Phi(D_2)|\w''\rangle\\ \hline \end{array} , \]
{where the sum extends over all possible enhancements $\w'$ of the intersection of $D_1$ with $D_2$.}
\end{Theorem}

\section{Examples }\label{examples}
\subsection{Examples derived from racks and quandles}\label{rackquandles}

\subsubsection{Rack and quandle invariants of links}
Recall that a rack $R$ is given by a set $R$ together with two (by the axioms not independent) operations $(x,y) \in R \times R\mapsto x\tr y \in R$  and $(x,y)\in R\times R\mapsto x \tl y \in R$, such that for each $x,y,z\in R$:{
\begin{enumerate}
\begin{minipage}{0.5 \textwidth}
 \item $x\tr ( y \tl x)=y,$
\item $(x \tr y)\tl x=y,$
\end{minipage}
\begin{minipage}{0.5 \textwidth}
\item  \label{condp}$x \tr (y \tr z)=(x \tr y) \tr (x \tr z),$
\item \label{cond} $(x \tl y) \tl z=(x \tl z) \tl (y \tl z)$.
\end{minipage}
\end{enumerate}}
A quandle $Q$ is a rack satisfying {the extra condition:
 $x\tl x=x=x \tr x$, $\forall x \in Q$.}

There is a more economical definition of a rack: it is a set $R$ with an operation $(x,y) \mapsto x \tl y$, such that condition \eqref{cond} above holds and such that for each $y \in R$ the map $ x \mapsto x \tl y$ is bijective. {Its inverse } will give us the map $x \mapsto y \tr x$.
The following result and proof appeared in \cite{N}.
\begin{lemma}[Nelson Lemma]\label{Nelson}
Given a rack $R$, the maps $x \mapsto x \tr x$ and $x \mapsto x \tl x$ are injective (thus bijective if the rack is finite.)
\end{lemma}
\begin{proof}
Let $x$ and $y$ {belong to} $R$. Then 
\[(x \tr x) \tr y=(x \tr x) \tr ( x \tr (y \tl x))=x \tr(x \tr ( y \tl x))=x \tr y.\]
If $x \tr x=y \tr y$ then
{$ x \tr x=(x \tr x) \tr x=( y \tr y) \tr x= y \tr x.
$}
Since $y \tr y=x \tr x$ then $y \tr y=y \tr x$. This implies $x =y$, since the map $x \mapsto y \tr x$ is bijective, its inverse being $x \mapsto x \tl y$. The proof for $x \mapsto x \tl x$ is analogous.
 \end{proof}

Given a knot diagram $D$, and a rack $R$, a rack colouring of $D$ is an assignment of an element of $R$ to each arc of $D$, which at each crossing of the projection has the form shown in figure \ref{rackcolouring}.
\begin{figure}
\centerline{\relabelbox 
\epsfysize 2cm
\epsfbox{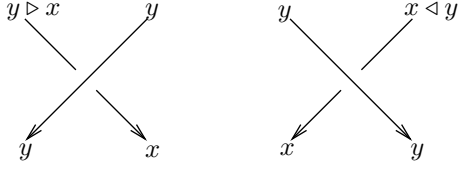}
\relabel{A}{${y \tr x}$}
\relabel{B}{${y}$}
\relabel{C}{${y}$}
\relabel{D}{${x }$}
\relabel{E}{${y}$}
\relabel{F}{${x \tl y}$}
\relabel{G}{${x}$}
\relabel{H}{${y}$}
\endrelabelbox}
\caption{A rack colouring of a link diagram in the vicinity of a vertex.
}
\label{rackcolouring}
\end{figure}

The following is well known. 
\begin{Theorem}\label{rli}
 Let $R$ be a finite rack. Then the number $I_R(D)$ of rack colourings of a link diagram $D$ is invariant under the Reidemeister moves 1', 2 and 3, and is therefore an invariant of framed links, also called $I_R$. Moreover if $R$ is a quandle then the number of rack colourings is invariant under the Reidemeister 1 move, therefore defining a link invariant.
\end{Theorem}

For a proof see \cite{N} or \cite{FR}. (Invariance under the Reidemeister moves 2 and 3 is immediate. Invariance under the Reidemeister 1' move, which is less trivial, is a consequence of the Nelson Lemma.)

\subsubsection{Reidemeister pairs derived from racks and quandles}\label{rpdrq}
Let us see that rack and quandle invariants can be written in the framework of this article. Let $R$ be a rack, which we suppose to be finite. Consider an arbitrary group structure on $R$. Call the group $G$. We do not impose any compatibility relation with the rack operations, we just assume that the underlying set of $R$ coincides with the underlying set of $G$. Consider the crossed module {$\Gc= (\id\colon G  \to G, {\rm ad})$,} where ${\rm ad}$ denotes the adjoint action of $G$ on $G$.  Put:
\begin{align}
 \psi({b},{a} )&={b}\,\,  {a}  \,\, {b}^{-1} ({b} \tr {a})^{-1}\,\,,  &
 \phi({b}, {a})&= {a}\,\,{b} \,\,({a} \tl {b})^{-1} \,\,{b}^{-1}.
\end{align}
\begin{Theorem}\label{racktophi}
 The pair $\Phi=(\p,\f)$ is a framed Reidemeister pair. Moreover $\Phi$ is an unframed Reidemeister pair if $R$ is a quandle.
\end{Theorem}
 \begin{proof}
  In this case relation $R{2}$ reads,
{$\phi({b}, {a}) \psi({b},{a}\tl {b})=1 $}
which follows tautologically. The relation $R{3}$ reads in this case:
\[
\phi({b},{a}) \,.\, {b} \phi({c},{a} \tl {b}) {b}^{-1} \,.\, \phi({c},{b}) = 
{a} \phi({c},{b}){a}^{-1} \,.\, \phi({c},{a}) \,.\, {c} \phi({b}\tl {c}, {a}\tl {c}) {c}^{-1},
\]
and follows easily, from relation \eqref{condp} of the definition of a rack. 

Let us now prove  relations 1. and 2. of Definition \ref{frp}.
Let ${a} \in G$. The equation ${z}=\partial(\phi({a},{z}))\,{a}$ means
{${z}={z}\,{a}\, ({z} \tl {a})^{-1} , $
or ${z} \tl {a}={a}$, that is ${z}={a} \tr {a}$.} By the Nelson Lemma \ref{Nelson}, for each ${z}$ the equation ${z}={a} \tr {a}$ has a unique solution $f({z}) \in R$. In this case $g({a})=\d(\psi({a},{a}))^{-1}{a}={a} \tr {a}$. Thus trivially $f\circ g=g \circ f=\id_R$. 

{Finally if $R$ is a quandle then $\psi(x,x)=x \, (x \tr x)^{-1}=x \, x^{-1}=1_E$.} \end{proof}

Since there is clearly, {by \eqref{psiphidef}, \eqref{Zforpsi} and \eqref{Zforphi}, a one-to-one correspondence between $\Gc$-colourings of a link diagram $D$ and rack colourings (with respect to $R$) of $D$, we have:}
\begin{Theorem}\label{rackex}
 Given a link diagram $D$,
{we have $I_\Phi(D)=I_R(D) 1_G,$}
where $1_G$ is the identity of $G$.
Therefore the class of invariants defined in this paper is at least as strong as {the class of rack link invariants.}
\end{Theorem}

There is a spin-off of the rack invariant in order to handle tangles. Given a rack $R$, recall that an $R$-enhanced tangle, Definition \ref{ent}, is a tangle together with a map from the boundary of $T$ into (the underlying set of) $R$. There is a category whose objects are the words $\w$ in $R \sqcup R^*$ and whose morphisms are $R$-enhanced tangles connecting them. Thus if the word $\w$ is the source of $T$ then $\w$ is a word having $i$ elements, where $i$ is the number of intersections of the tangle with $\R^2 \times \{1\}$. Moreover the $n^{\rm th}$ element of $\w$ is either the colour $a \in R$ given to the $n$-th intersection, or it is   $a^*$, the former happening if the strand is pointing downwards and the latter if the associated strand is pointing upwards. These conventions were explained in figure \ref{etangle}. Given an $R$-enhanced tangle $T$ let $\w(T)$ and $\w'(T)$ be the source and target of $T$, both words in $R \sqcup R^*$. Define $\langle \w(T) |I_R(T)| \w'(T) \rangle$ as being the number of  arc-colourings of a diagram of $T$ extending the enhancement of $T$. This defines an invariant of tangles, for any choice of colourings on the top and bottom of $T$ (in other words, for any $R$-enhancement of $T$). Clearly
\begin{Theorem}\label{tanglerack}
For any $R$-enhanced tangle $T$, putting $\w=\w(T)$ and $\w'=\w'(T)$ we have {(see definitions \ref{ev1} and \ref{ev2} for notation):}
\[ \langle \w(T) |I_\Phi(T)|\w'(T)\rangle =   \langle \w |I_R(T)|\w'\rangle \quad \begin{CD} e(\w)\\@VV e(\w') e(\w)^{-1}V \\ e(\w')\end{CD}.\]
\end{Theorem}

\subsubsection{Reidemeister pairs derived from rack and quandle cocycles}\label{qc}
We can extend the statement of Theorem \ref{rackex} for the case of rack cohomology invariants of knots. Let $R$ be a rack. Let $V$ be an abelian group. We say that a map $w \colon R \times R \to V$ is a rack 2-cocyle if:
\[w(x,y)+w(x \tl y,z)=w(x,z)+ w(x \tl z, y \tl z), \textrm{ for each}\, x,y,z \in R. \]
If $R$ is a quandle, such a $w$ is said to be a quandle cocycle if moroever $w(x,x)=0_V$, for each $x \in R$. For details see \cite{CJKLS,CJKLS2,N,E1,E2}.

Consider any group structure  $G$ on the set $R$, which may be completely independent of the 
rack operations. Consider the crossed module $(\partial \colon G \times V \to G,\bullet)$, where {$g \bullet (h,v)=(ghg^{-1},v), \textrm{ for each } g,h \in G \textrm{ and } v \in V,$} which is a left action of $G$ on $G \times V$ by automorphisms, and $\partial(g,v)=g$, for each $(g,v) \in G \times V$. 
Given a rack 2-cocycle $w\colon R \times R \to V$, set:
\begin{align*}
 \psi({b},{a} )&=\big ({b}\,\,  {a}  \,\, {b}^{-1} ({b} \tr {a})^{-1},w({b} \tr {a},{b})\big),
 &\phi({b}, {a})&= \big ({a}\,\,{b} \,\,({a} \tl {b})^{-1} \,\,{b}^{-1}, w({a},{b})^{-1}\big).
\end{align*}
\begin{Theorem}
 The pair $\Phi=(\psi,\phi)$ is a framed Reidemeister pair. Its associated framed link invariant coincides with the usual rack cohomology invariant of framed links. Morever if $R$ is a quandle and $w$ a quandle 2-cocycle then $\Phi=(\psi,\phi)$ is an unframed Reidemeister pair and its associated invariant  of links coincides with the usual quandle cocycle link invariants \cite{CJKLS}. 
\end{Theorem}
\begin{proof}
Analogous to the proof of Theorems \ref{racktophi} and \ref{rackex}.
\end{proof}

\noindent Therefore the class of invariants defined in this paper is at least as strong as the class of invariants of links derived from quandle cohomology classes.

\subsection{Relation with the Eisermann knot invariant}\label{ep}

\subsubsection{String knots, long knots, knot meridians and knot longitudes}\label{sklk}
Recall that an (oriented) long knot is an embedding  $f$ of  $\R$ into $\R^3$ such that, for sufficiently large (in absolute value) $t$, we have {$f(t)=(0,0,-t)$.} These are considered up to isotopy with compact support. Clearly long knots (up to isotopy) are in one-to-one correspondence with isotopy classes of tangles whose underlying 1-manifold is the interval, and whose boundary is $\{0\} \times \{0\} \times \{\pm 1\}$, being, furthermore, oriented downwards. These are usually called string knots.

There exists an obvious {closing map, $\rm cl$,} sending a string knot $L$ to a closed knot ${\rm cl}(L)$. It is well known that this defines a one-to-one correspondence between isotopy classes of string knots and isotopy classes of oriented knots. To see this, note that a map sending a closed knot $K$ to a long knot $L_K$ can be obtained by choosing a base point $p \in K$. Then there exists an (essentially unique) orientation  (of $S^3$ and of $K$) preserving  diffeomorphism $(S^3\setminus \{p\},K \setminus \{p\}) \to (\R^3,L_K^p)$, where $L_K^p$ is a long knot with ${\rm cl}(L_K^p)=K$. Note that $L_K^p$ depends only on the orientation preserving diffeomorphism class of the triple $(S^3,K,p)$, thus since all pairs $(K,p)$, with fixed $K$, but arbitrary $p$, are isotopic we can see that 
$L_K^p$ depends only on $K$, thus we can write it as $L_K$.
\begin{figure}
\centerline{\relabelbox 
\epsfysize 1.7cm 
\epsfbox{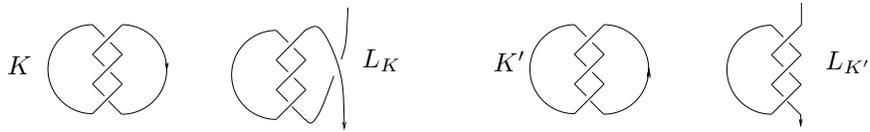}
\relabel{K}{${K}$}
\relabel{LKK}{${L_{K'}}$}
\relabel{LK}{${L_{K}}$}
\relabel{KK}{${K'}$}
\endrelabelbox}
\caption{Turning an oriented knot into a string knot in two different cases.\label{LinkTangle}
}
\end{figure}

Let $D$ be a knot diagram of the knot $K$. Consider the Wirtinger generators of the fundamental group $\pi_1(C_K)$ of the complement $C_K=S^3 \setminus n(K)$ of $K$ {(here $n(K)$ is an open regular neighbourhood of $K$)}; we thus have a meridian for any arc of the diagram\ $D$. Let $a$ be an arc of $D$ and $p$ a base point of $K$ in $a$.  Then there is a meridian $m_p=m$ of $K$ encircling {the arc $a$ at $p$}, whose direction is determined by the right hand rule. Let $D^2=\{z \in \mathbb{C}: |z|\leq 1\}$ and $S^1 = \partial D^2$. Choose an embedding $f\colon S^1\times D^2 \to S^3$ such that 
\begin{itemize}
 \item $f(S^1 \times \{0\})=K$, preserving orientations, with $f(1,0)=p$.
\item $f(\{1\} \times S^1)=m$
\item $f(S^1 \times \{1\})$ has zero linking number with $K$.
\end{itemize}
If we take  $f(1,1)$ to be the base point of $S^3$, then the homotopy class $l_p=l\in \pi_1(C_K)$  of   $f(S^1 \times \{1\})$ is called a longitude of $K$ {\cite{BZ}}. It is well known that the triple $(\pi_1(C_K), m,l)$, considered up to isomorphism, is a complete invariant of the knot $K$ \cite{W}. Note that if we choose another base point $p'$ of $K$ then $m_{p'}$ and $l_{p'}$ can be obtained from $m_p$ and $l_p$ by conjugating by a single element of $\pi_1(C_K)$.

The longitude $l_p$, being an element of the fundamental group of the complement {$C_K$ of $K$}, can certainly be expressed in terms of the Wirtinger generators. This can be done in the following way;  for details see \cite{E1,E2}. Let $a_0=a$ be the arc of the diagram $D$ of $K$ containing the base point of $K$. Then go around the knot in the direction of its orientation. This makes it possible to order the arcs of $K$, say as  $a_1,a_2,\dots, a_n$;  we would have $a_n=a_1$, except that we prefer to see $K$ as being split at the base point $p$, separating the arc $a$ in two. We can also order the crossings of $D$.

 The longitude $l_p$ of $K$ is expressed as a product of all elements of {$\pi_1(C_K)$, associated to the arcs we undercross} as we travel from $p$ to $p$, making sure that the linking number of $l_p$ with $K$ is zero. {Therefore any arc $a_i$ has also assigned a partial longitude $l_i$ (the product of the elements  of {$\pi_1(C_K)$, associated to the arcs we undercross,  as we travel from $p$ to $a_i$)}. We thus have $l_n=l$.} 
Given an arc  $a$ of $D$ denote the corresponding element of the fundamental group of the complement  by  $g_a$. Then clearly we have {that $g_{a_i}=l_i^{-1} g_{a_1} l_i.$}  The way to pass from  $l_i$ to $l_{i+1}$ appears in figure \ref{partial} for the positive and negative crossing. 

\begin{figure}
\centerline{\relabelbox 
\epsfysize 4cm
\epsfbox{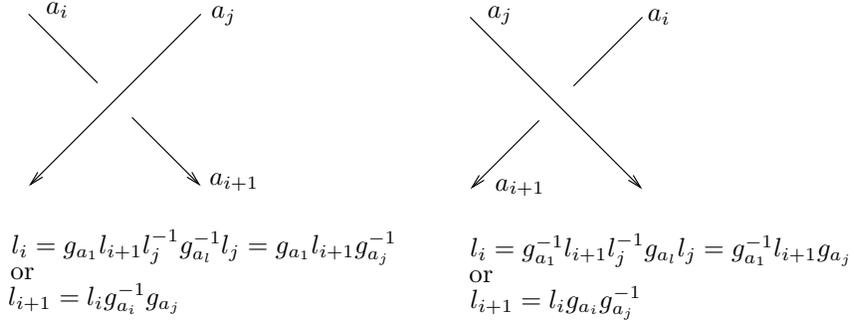}
\relabel{a}{${a_j}$}
\relabel{l}{${a_i}$}
\relabel{ll}{$a_{i+1}$}
\relabel{b}{${a_j}$}
\relabel{k}{${a_i}$}
\relabel{kk}{$a_{i+1}$}
\relabel{M}{$l_{i}=g_{a_1}l_{i+1}l_{j}^{-1}g_{a_l}^{-1}l_{j}=g_{a_1}l_{i+1}g_{a_j}^{-1}$}
\relabel{MM}{$l_{i+1}=l_i g_{a_{i}}^{-1}g_{a_j} $}
\relabel{N}{$l_{i}=g_{a_1}^{-1}l_{i+1}l_{j}^{-1}g_{a_l}l_{j}=g_{a_1}^{-1}l_{i+1}g_{a_j}$}
\relabel{NN}{$l_{i+1}=l_i g_{a_{i}}g_{a_j}^{-1} $}
\relabel{A}{or}
\relabel{P}{or}
\endrelabelbox}
\caption{Rules for partial longitudes at crossings.}
\label{partial}
\end{figure}

Given $i$, let $j_i$ be the number of the arc {splitting $a_i$ and $a_{i+1}$. Let $\theta_i$ be the sign of the $i$-th crossing. Then:}
\begin{equation}\label{longformula}
 l=\prod_{i=1}^{n-1} g_{a_i}^{-\theta_i}g_{a_{j_i}}^{\theta_i}=\prod_{i=1}^{n-1} l_i^{-1}g_{a_1}^{-\theta_i}l_i l_{j_i}^{-1} g_{a_1}^{\theta_i} l_{j_i}=\prod_{i=1}^{n-1}[l_i^{-1},g_{a_1}^{-\theta_i}]\,\,[{g_{a_1}^{-\theta_i}},l_{j_i}^{-1}];
\end{equation}
more generally, if $k \in \{2,\dots,n\}$:
\begin{equation}\label{partiallongformula}
 l_k=\prod_{i=1}^{k-1} g_{a_i}^{-\theta_i}g_{a_{j_i}}^{\theta_i}=\prod_{i=1}^{k-1} l_i^{-1}g_{a_1}^{-\theta_i}l_i l_{j_i}^{-1} g_{a_1}^{\theta_i} l_{j_i}=\prod_{i=1}^{k-1}[l_i^{-1},g_{a_1}^{-\theta_i}]\,\,[{g_{a_1}^{-\theta_i}},l_{j_i}^{-1}].
\end{equation}
{Thus both the longitude $l$ and any partial longitude $l_k$ belong to the commutator} subgroup of the fundamental group of the complement of $K$.
Also
$l_{i+1}= l_i [l_i^{-1},g_{a_1}^{-\theta_i}]\,\,[{g_{a_1}^{-\theta_i}},l_{j_i}^{-1}].$

{ In remark \ref{longformulanew} we will present another formula for a knot longitude.}

\subsubsection{The Eisermann invariant of knots} Let $K$ be a knot in $S^3$. 
 Consider the fundamental group of the complement  $C_K=S^3 \setminus n(K)$ of the knot $K$. Here $n(K)$ is an open regular neighbourhood of $K$. Choose a base point $p$ of $K$. Let the associated meridian and longitude of $K$ in $\pi_1(C_K)$ be denoted by $m_p$ and $l_p$, respectively. Note that $[m_p,l_p]=1$.

Let $f\colon \pi_1(C_K) \to G$ be a group morphism. {Therefore} $f(l_p)\in G'\doteq [G,G]$, the derived (commutator) group of $G$, {generated by the commutators $[g,h]\doteq ghg^{-1}h^{-1}$.} Moreover {$[f(l_p),f(m_p)]=1_G$.} Then \[f(l_p)\in \Lambda\doteq [G,G] \cap C(x),\]
where {$x=f(m_p)$} and $C(x)$ is the set of elements of $G$ commuting with $x$.

Let $G$ be a finite group. Let $x$ be an element of $G$. The Eisermann invariant  \cite{E2} (also called Eisermann polynomial) is:
\[{E(K)=\sum_{\left \{f\colon \pi_1(C_K) \to G\,|\, f(m_p)=x\right \}} f(l_p) \in \N(\Lambda). }\]
Note that if we choose a  different base point $p'$ of $K$ then $E(K)$ stays invariant since $m_{p'}=h\,m_p\,h^{-1}$ and $l_{p'}=h\,l_p\, h^{-1}$, for some common {$h \in \pi_1(C_K)$}. The Eisermann invariant can be used to detect chiral and non invertible knots  \cite{E2}.

Clearly $E(K)$ is given by a map $f_E^K\colon G' \to \N$, where 
\[E(K)=\sum_{g \in G'}  f_E^K(g) g. \]
Note that $f_E^K(g)=0$ if $g \not \in  \Lambda$.

Let us see that the Eisermann invariant can be addressed using Reidemeister pairs. This is  a consequence of 	the previous subsections and the discussion in \cite{E2,E1}, which we closely follow, having discussed and completed the most relevant issues in \ref{sklk}.

Let $G$ be a group. Choose $x\in G$ and consider from now on the pair $(G,x)$. We will use the notation $h^g=g^{-1} h g$, where $g,h \in G$.
Let \begin{align*}
Q&=\big\{x^g, g \in G'\big\} \subset G,
&\overline{Q}&=\big\{x^g, g \in G\big\} \subset G.
\end{align*}

\begin{lemma}
 Both sets $Q$ and $\overline{Q}$ are self conjugation invariant:
\[{a,b \in Q \implies a^{-1} b a \in Q \quad \quad \textrm{ and }\quad \quad  a,b \in \overline{Q} \implies a^{-1} b a \in \overline{Q}}\]
Therefore $Q$ and $\overline{Q}$ are both quandles, with quandle operation $h\tl g=h^g$. 
\end{lemma}
\begin{proof}
 Given $g,h \in G'$ we have $$(g^{-1} xg )^{-1} (h^{-1} x h)  (g^{-1} xg )=x^{x^{-1}hg^{-1} x g},$$
and $$x^{-1}hg^{-1} x g=x^{-1}hg^{-1} x gh^{-1} h=[x^{-1},hg^{-1}] h \in G'.$$ 
The proof for $\overline{Q}$ is analogous.
\end{proof}

It is easy to see that:
\begin{lemma}[Eisermann] \label{eqd}
Let $G$ be a group. Given arbitrary $x \in G$, both $G'$ and $G$ are quandles, with quandle operation: \begin{align} h\tl g&=x^{-1}hg^{-1} x g\,\,, &g \tr h'&=xh'g^{-1}x^{-1} g. \end{align}
There are also  quandle maps $p\colon G' \to Q$ and $\overline{p}\colon \overline{Q} \to G$ with $p(g)=x^g$.
\end{lemma}
Recall the rack  invariant of tangles, defined just before Theorem \ref{tanglerack}.
\begin{Theorem}[Eisermann]
For any  knot $K$ and any $g$ in $G'$ it holds that:
\[\langle 1_G |I_{G'}({D_K})| g\rangle =f^K_E(g), \]
where $D_K$ is any string knot diagram associated to $K$. Of course we regard $D_K$ as {a $G$-enhanced} tangle diagram coloured with $1_G=1_{G'}$  at the top and with $g$ at the bottom.
\end{Theorem}
\begin{proof}
 Follows from the discussion in \ref{sklk} and especially figure \ref{partial}.
\end{proof}

\begin{remark}
 The previous theorem is also valid for the quandle structure in  $G$, Lemma \ref{eqd}. Given the form of the quandle is easy to see that if $g,h \in G'$:
\[\langle g |I_{G'}({D_K})|h\rangle=\langle g |I_{G}({D_K})|h\rangle.\]
\end{remark}

By using subsection \ref{rackquandles} we will show that the Eisermann invariant can be addressed in our framework, by passing to string knots. 
 Suppose we are given  a finite group $G$ and $x \in G$ {(it may be that $x \in G \setminus G'$)}.
  We can choose any group operation structure in the underlying set of $G'$. We take the most obvious one, given by the inclusion $G' \subset G$.  The associated crossed module is $G'\stackrel{\rm id}{\rightarrow} G'$, with $G'$ acting on itself by conjugation.

 {Define, given $g,h \in G'$, the pair  $\Phi^x=(\psi^x,\phi^x)$,  as:}
\begin{equation}
%\label{phipidefinition}
\begin{split}
\phi^x(g,h)&= hg(x^g)^{-1} x^h h^{-1}g^{-1}=hx^{-1}gh^{-1}x g^{-1}=[hx^{-1},gx^{-1} ]\\
\psi^x(g,h)&=[g,h][hg^{-1},x]=[xhg^{-1}x^{-1}gx^{-1},gx^{-1}]^{-1}
\end{split}
\label{eisphipsi}
\end{equation}
\begin{Theorem}
 The pair $\Phi^x=(\psi^x,\phi^x)$ is an unframed Reidemeister pair for the crossed module $G'\stackrel{\rm id}{\rightarrow} G'$, with $G'$ acting on itself by conjugation. Let $K$ be an oriented knot and $L_K$ be the associated string knot.  Given $g \in G'$ then
\begin{equation}\label{eqinv}\langle 1_{G'} | I_{\Phi^x}(L_K)| g\rangle=f^K_E(g).\end{equation}
\label{eisermann}
\end{Theorem}
\begin{proof}
The expressions for {$\Phi^x=(\psi^x,\, \phi^x)$} guarantee that the colourings of the arcs at a crossing are those given by the Eisermann quandle operation and its inverse (Lemma \ref{eqd}). Thus  equation \eqref{eqinv} holds, and it is enough to check that $\Phi$ does indeed satisfy the conditions to be a unframed Reidemeister pair. Clearly the Reidemeister 1 condition \rref{R1} holds: {$\psi^x(l,l)=1$}.  The Reidemeister 2 equation \rref{R2} is:
{$
\phi^x({l},{m})\, \psi^x({l}, x^{-1}{m}{l}^{-1}x{l}) =1
$,}
i.e.
\[
[{m}x^{-1},{l}x^{-1}]\, [{l},x^{-1}{m}{l}^{-1}x{l} ]\, [x^{-1}{m}{l}^{-1}x,x]=1.
\]
Writing this out in full, one obtains:
\begin{multline*}
{m}x^{-1}{l}x^{-1}x{m}^{-1}x{l}^{-1} \, . \, {l}(x^{-1}{m}{l}^{-1}x{l}){l}^{-1}({l}^{-1}x^{-1}{l}{m}^{-1}x) \, . \,   (x^{-1}{m}{l}^{-1}x)x(x^{-1}{l}{m}^{-1}x)x^{-1} \\
  = {m}\, \underline{x^{-1}{l}{m}^{-1}x{l}^{-1} \, . \, {l} x^{-1} {m}{l}^{-1}x} \, \,\,
\underline{{l}^{-1}x^{-1} {l} {m}^{-1} x\, . \, x^{-1} {m}{l}^{-1} x {l} }\, {m}^{-1} = 1.
\end{multline*}
In the above computation, {the underlined factors are} equal to 1.

To avoid confusion with the quandle operation $\tr$, the left action of $G'$ on $G'$ by conjugation will be denoted by $g \bullet h$, thus $g \bullet h=ghg^{-1}$ for each $g,h \in G'$. 
The Reidemeister 3 equation \rref{R3}
for this case reads:
\[
\phi^x({l},{m})\, . \, {l}\bullet \phi^x({n},{p}) \, . \, \phi^x({n},{l}) = {m}\bullet \phi^x({n},{l}) \, . \,\phi^x({n},{m}) \, . \, {n}\bullet \phi^x({r},{q}),
\] 
{where
$
{p}=x^{-1}{m}{l}^{-1}x{l}, \quad {r}= x^{-1}{l}{n}^{-1}x{n}, \quad {q}=x^{-1}{m}{n}^{-1}x{n}.
$}
The left-hand side of the above equation, written out in full, is:
\begin{eqnarray*}
\lefteqn{{m}x^{-1}{l}{m}^{-1}x{l}^{-1} \, . \, {l}({p}x^{-1}{n}{p}^{-1}x{n}^{-1}){l}^{-1}  \, . \, ({l}x^{-1}{n}{l}^{-1} x{n}^{-1} )} \\
& & \quad \quad \quad= {m}\, \underline{x^{-1}{l}{m}^{-1}x \, . \, x^{-1} {m}{l}^{-1}x } \, {l} x^{-1} {n} ({l}^{-1}x^{-1}{l}{m}^{-1}x)x  
{n}^{-1} \, . \, x^{-1} {n} {l}^{-1} x {n}^{-1} \\
& &  \quad \quad \quad   = {m}{l}{n} ({n}^{-1}x^{-1}{n}{l}^{-1}x^{-1}{l}{m}^{-1}x^2) ({n}^{-1}x^{-1}{n}{l}^{-1}x){n}^{-1},
\end{eqnarray*}
and the right-hand side leads to the same expression:
\begin{eqnarray*}
\lefteqn{{m}({l}x^{-1}{n}{l}^{-1}x{n}^{-1}){m}^{-1} \, .  \,  {m}x^{-1}{n}{m}^{-1}x{n}^{-1}  \, . \,   {n}({q}x^{-1}{r}{q}^{-1}x{r}^{-1}){n}^{-1} }\\
& & = {m} {l}x^{-1}{n}{l}^{-1}  \,  \underline{x{n}^{-1}   x^{-1}{n}{m}^{-1}x  (x^{-1} {m}{n}^{-1}x {n}) \, . \, x^{-1} } \,\,\, (x^{-1}{l} \, \underline{{n}^{-1} x {n})( {n}^{-1} x^{-1} {n} }\, {m}^{-1}x)   x({n}^{-1}x^{-1}{n}{l}^{-1}x){n}^{-1} \\
& & =  {m}{l}{n} ({n}^{-1}x^{-1}{n}{l}^{-1}x^{-1}{l}{m}^{-1}x^2) ({n}^{-1}x^{-1}{n}{l}^{-1}x){n}^{-1}.
\end{eqnarray*}
In both derivations we insert the definitions of ${p},\, {r}$ and ${q}$ in the {1st} equality, and regroup the factors after eliminating the underlined expressions in the {2nd} equality. Note that both sides equal ${m}{l}{n}{u}^{-1}{r}^{-1}{n}^{-1}$, which is the product of the colourings of the 6 external arcs in Figure \ref{R3Eisermann} for Reidemeister 3, taken in anticlockwise order starting with ${m}$, where ${u}$ is the colouring assigned to the rightmost upper arc, i.e. 
${u}= x^{-1}{p}{n}^{-1}x{n}=x^{-1} {q}{r}^{-1}x{r}$.
\begin{figure}
\centerline{\relabelbox 
\epsfysize 3cm
\epsfbox{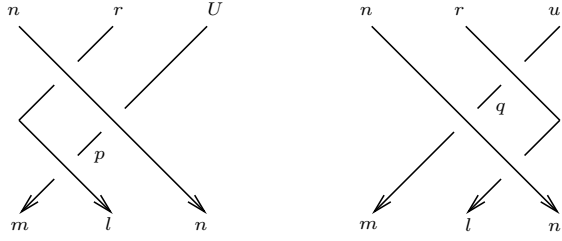}
\relabel{M}{$\scriptstyle{{m}}$}
\relabel{L}{$\scriptstyle{{l}}$}
\relabel{N}{$\scriptstyle{{n}}$}
\relabel{NN}{$\scriptstyle{{n}}$}
\relabel{R}{$\scriptstyle{{r}}$}
\relabel{U}{$\scriptstyle{U}$}
\relabel{P}{$\scriptstyle{{p}}$}
\relabel{MM}{$\scriptstyle{{m}}$}
\relabel{LL}{$\scriptstyle{{l}}$}
\relabel{NNN}{$\scriptstyle{{n}}$}
\relabel{NNNN}{$\scriptstyle{{n}}$}
\relabel{RR}{$\scriptstyle{{r}}$}
\relabel{UU}{$\scriptstyle{{u}}$}
\relabel{Q}{$\scriptstyle{{q}}$}
\endrelabelbox}
\caption{Two sides of the Reidemeister-III move in the proof of Theorem \ref{eisermann}.}
\label{R3Eisermann}
\end{figure}
\end{proof}

Consider the quandle structure in $G$, Lemma \ref{eqd}, given by the same formulae as the one of $G'$. Consider the crossed module given by the identity map $G \to G$ and the adjoint action of $G$ on $G$.  Given $x \in G$ we have an unframed Reidemeister pair  $\bar{\Phi}^x=(\bar\psi^x,\bar \phi^x)$, with the same formulae as \eqref{eisphipsi}, namely: 
 \begin{equation}\label{epb}
\begin{split}
\bar{\phi}^x(g,h) &= hg(x^g)^{-1} x^h h^{-1}g^{-1}=hx^{-1}gh^{-1}x g^{-1}=[hx^{-1},gx^{-1}]\,\,,\\
\bar{\psi}^x(g,h)&=[g,h][hg^{-1},x]=[xhg^{-1}x^{-1}gx^{-1},gx^{-1}]^{-1}
\end{split}
\end{equation}
{for  $g,h \in G$.} {It is easy to see that:}
\begin{proposition}
 Let $a,b \in G'$. {For each $x \in G$, and  each string knot $L_K$:}
\[\langle a | I_{\Phi^x}(L_K)| b\rangle =\langle a  | I_{\bar{\Phi}^x} (L_K)| b\rangle\]
\end{proposition}

\begin{remark}[Formula for a knot longitude]\label{longformulanew}
Our approach for defining Eisermann invariants, and the proof of Theorem \ref{eisermann}, provides a different formula to \eqref{longformula} for the longitude {$l_p=l$} of a knot $K$, if $K$ is presented as the closure of a string knot $L$. It is assumed that the base point ${p \in K}$ of the closed knot $K$ lives in the top end of $L$.  Let $G=\pi_1(C_K)$. Consider the crossed module $(\id\colon G \to G,\ad)$. Let $x$ be the element of $G$ given by the top strand  $a$ of $L$. Let $b$ be the bottom strand {of $L$}.  Consider the Reidemeister pair $\bar{\Phi}^x$ in \eqref{epb}, thus
\begin{equation}
{\bar{\phi}^x(g,h)=[hx^{-1},gx^{-1}] \quad \quad \quad  \textrm{ and }\quad\quad  \quad  \bar{\psi}^x(g,h)={[g,h][hg^{-1},x]}}={[xhg^{-1}x^{-1}gx^{-1},gx^{-1}]^{-1} .}
\end{equation}
 Consider a diagram $D$ of $L$.
 Colour each arc $c$ of the diagram $D$  with the corresponding partial longitude $l_c$, as defined in \ref{sklk}. Therefore $l_a=1$ and $l_b=l$. Then by the proof of Theorem \ref{eisermann} one has a Reidemeister colouring $F$. If we evaluate $F$ {(definitions \ref{ev1} and  \ref{ev2})}, we have a morphism $l_a \ra{e(F)} l_b$,  hence {$e(F)=l=l_p$.} The form of $e(F)$ thus yields an alternative formula for the knot longitude, which will be crucial {for giving a homotopy interpretation of the lifting (Theorem \ref{liftextension2}) of the Eisermann invariant.}
\end{remark}

\subsubsection{One example of Eisermann invariants}
Let $G$ be a group with a base point $x$.
The explicit calculation of the invariant $\langle a |I_{\bar{\Phi}^x}(K_+)| b \rangle $ and $\langle a |I_{\bar{\Phi}^x}(K_-)| b \rangle$
for the trefoil knot $K_+$ and its mirror image $K_-$ (the positive and negative trefoils), converted to string knots,  appears in {figure \ref{TM}.} In particular, given $a,g,h \in G$, we have:
\begin{align*}
\langle a |I_{\bar{\Phi}^x}(K_+)| g \rangle&= \# \big\{h \in G\colon   x^{3}hg^{-1}
x^{-1}gh^{-1}x^{-1}hg^{-1} x^{-1} g=a \, ; \,   x^{2}gh^{-1}x^{-1}hg^{-1}x^{-1}g=h \big\}\,\,,\\ \langle a |I_{\bar{\Phi}^x}(K_-)| h \rangle&=\# \big\{g \in G\colon   x^{-3}hg^{-1}xgh^{-1}xhg^{-1} x g=g \, ; \,   x^{-2}gh^{-1}xhg^{-1}xg=a \big\} \, . 
\end{align*}
\begin{figure}
\begin{minipage}{0.5\textwidth}
\centerline{\relabelbox 
\epsfysize 4.3cm 
\epsfbox{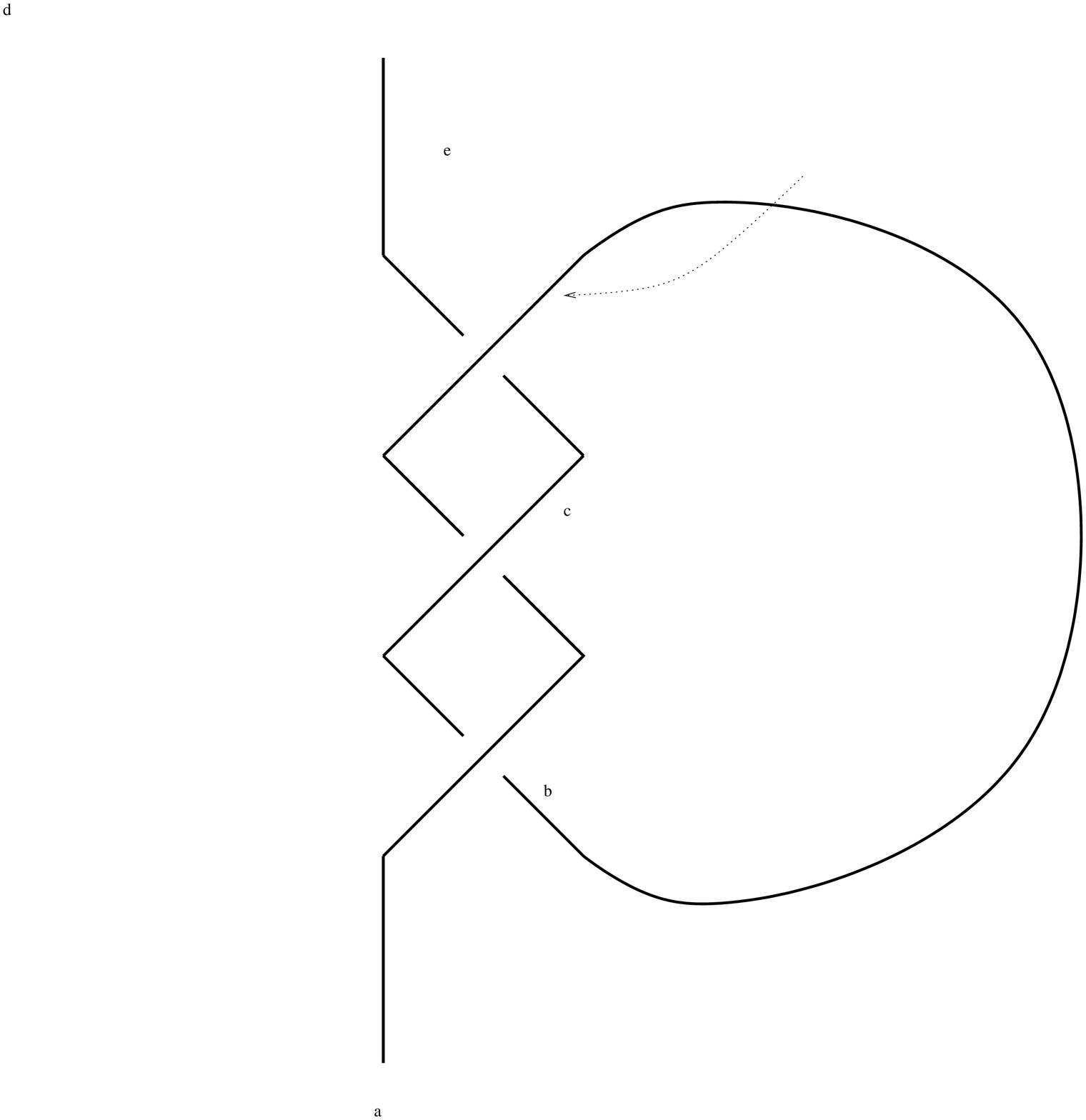}
\relabel{b}{${h}$}
\relabel{a}{${g}$}
\relabel{c}{${xhg^{-1}x^{-1}g}$}
\relabel{e}{${x^{2}gh^{-1}x^{-1}hg^{-1}x^{-1}g}$}
\relabel{d}{$x^{3}hg^{-1}x^{-1}gh^{-1}x^{-1}hg^{-1} x^{-1} g $}
\endrelabelbox}
\end{minipage}
\begin{minipage}{0.5\textwidth}
\centerline{\relabelbox 
\epsfysize 4.3cm 
\epsfbox{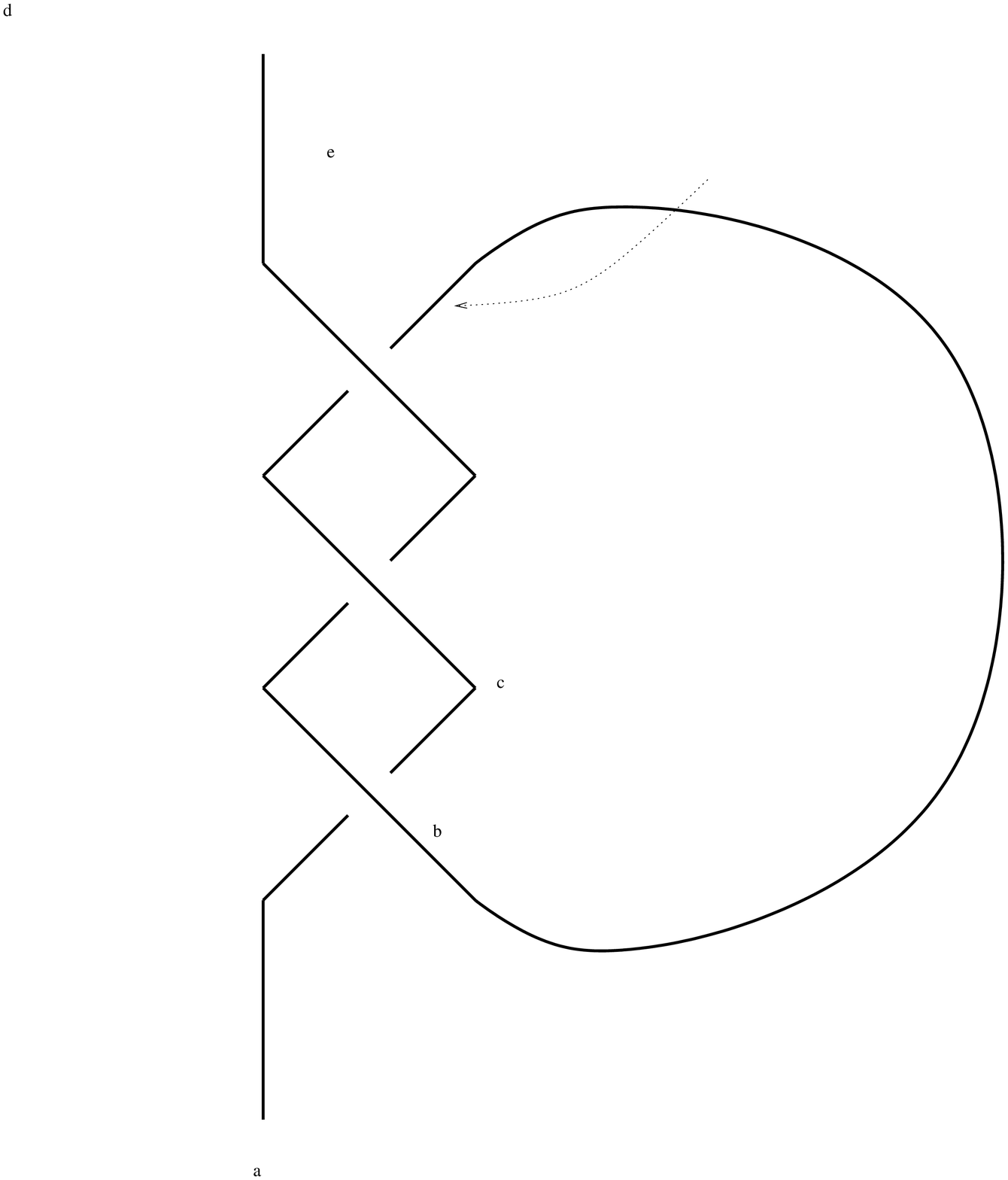}
\relabel{a}{${h}$}
\relabel{b}{${g}$}
\relabel{c}{${x^{-1}hg^{-1}xg}$}
\relabel{d}{${x^{-2}gh^{-1}xhg^{-1}xg}$}
\relabel{e}{$x^{-3}hg^{-1}xgh^{-1}xhg^{-1} x g $}
\endrelabelbox}
\end{minipage}
\caption{{The Eisermann polynomial for the positive and negative trefoil knots $K_+$ and $K_-$}\label{TM}}
\end{figure}

Consider from now on $G=S_5$. We refer to   table \ref{T}, displaying the values of $\langle 1 |I_{\bar{\Phi}^x}(K_-)| a \rangle$ and  $\langle 1 |I_{\bar{\Phi}^x}(K_+)| a \rangle$ for some choices of $x \in S_5$, representing all possible conjugacy classes in $S_5$. In each {2nd} and {3rd} {row entry} of table \ref{T}, we {put:} \[{\sum_{h \in S_5} \langle 1 |I_{\bar{\Phi}^x}(K_-)| h\rangle \quad \quad \textrm{ and } \quad \quad  \sum_{g \in S_5} \langle 1 |I_{\bar{\Phi}^x}(K_+)| g\rangle,}\] both elements of the group algebra of $S_5$.
\begin{table}
\begin{tabular}{|l|l|l|l|l|l|l|l|}
 $x$  & $\id_{}$ &$(12)$         & $(12)(34)$       & ($123)$         & $(123)(45)$ & $(1234)$  & $(12345)$ \\ \hline
 $K_-$ & $\id_{}$& $7\id_{}$& $5\id_{}$   & $7\id_{}$  &$\id_{}$& $\id_{}+4(13)(24)$   &$\id_{}+5(12345)$  \\                      \hline 
 $K_+$ & $\id_{}$& $7\id_{}$& $5\id_{}$   & $7\id_{}$  &$\id_{}$ & $\id_{}+4(13)(24) $ & 
$\id_{}+5(15432)$ 
\end{tabular}
\vskip 0.2 cm
\caption{\label{T}{The Eisermann invariant for the negative and positive trefoils for $G=S_5$.}}
\end{table}
{Therefore} the invariant  $L \mapsto \langle 1 |I_{\bar{\Phi}^x}(L)| b \rangle$, where we identify a knot with its associated string knot, {separates} the trefoils, for $x=(12345)$. {This  is due to Eisermann \cite{E2}. }

\subsection{{Reidemeister pairings derived from central extensions of {groups} - lifting the Eisermann invariant}}\label{LEI}

{Recall the construction of the Eisermann invariant  in our framework (subsection \ref{ep}). The quandle underlying the Eisermann invariant corresponds to the Reidemeister pair given in equation 
\eqref{eisphipsi}.}
\begin{definition}[{Unframed Eisermann lifting} ] Let $G$ be a {finite} group and $x \in G$.
An unframed Eisermann lifting is given by a crossed module $(\d: E\rightarrow G, \tr)$, and an unframed Reidemeister pair $\Phi^x=(\phi^x, \psi^x)$,  where $\phi^x, \, \psi^x: G \times G \rightarrow E$, such that, {given $L,M\in G$}: 
\begin{align*}
\d(\phi^x(L,M)) &=  [Mx^{-1},Lx^{-1} ]\,\,, & \d(\psi^x(L,M)) = & [L,M][ML^{-1},x].
\end{align*}
\end{definition}
The colourings of the arcs of any tangle diagram will correspond to those given by the Eisermann quandle, but there may be additional information contained in the assignments of elements of $E$ to the crossings of the diagram, i.e. an Eisermann lifting is a refinement of the Eisermann invariant. 

In this subsection we will construct an unframed Eisermann {lifting} from each central extension of groups:  \[\{0\}\to A \to E \ra{\d} G \to \{1\}.\] Here $G$ is a finite group and { $\d \colon E \to G$ is} a surjective group map, such that the kernel $A$ of {$\d$ is central in $E$}.

Choose an arbitrary section $s \colon G \to E$ of $\d$, meaning $\d(s(g))=g$, for each $g \in G$. Therefore $s(gh)=s(g)\,s(h)\,\lambda(g,h)$, where $\lambda(g,h)$ is in the centre of $E$, for each $g,h \in G$. Moreover given $e \in E$ then $s(\d(e))=e\, c(e)$, where $c(e)$ is in the centre of $E$. This is because $\d(s(\d(e))=\d(e)$. Clearly:
\begin{lemma}\label{liftextension}
 {The map $(g,e) \in G \times E \mapsto g \tr e=s(g) \,\, e\,\, s(g)^{-1}\in E$ is a left action of $G$ on $E$ by automorphisms, and with this action $\d\colon E \to G$ is a crossed module. Moreover, the action $\tr$ does not depend on the section $s$.}
\end{lemma}

{Given a section $s\colon G \to E$ of $\partial\colon E \to G$,} define, for each $g,h \in G$:
\begin{equation}\label{dpl}
 \{g,h\}=[s(g),s(h)]. 
\end{equation}
{(This does not depend on the chosen section $s$ of $\d$ since $\ker(\partial)$ is central in $E$.)}
\begin{lemma}
{For each $a,b \in E$ we have $[a,b]=\{\d(a),\d(b)\} .$}
\label{bracketlift}
\end{lemma}
\begin{proof}
{Given $a,b \in E$  we have $\{\d(a),\d(b)\}=[s(\d(a)),s(\d(b)]=[a\,c(a),b\,c(b)]=[a,b]. $}
\end{proof}

\begin{Theorem}\label{liftextension2} Let $G$ be a {finite}  group and $x \in G$. Let $\d\colon E \to G$ be a surjective group morphism such that the kernel $A$ of $\partial$ is central in $E$. The pair $\Phi^x =(\phi^x, \psi^x)$, given by:
\begin{align*}
\phi^x(g,h)&=\{hx^{-1}, gx^{-1}\} &\psi^x(g,h) &= \{g,h\} \{hg^{-1},x\};
\end{align*}
is an unframed Eisermann lifting for the crossed module $(\d \colon E \to G, \tr)$, of Lemma \ref{liftextension}.
\end{Theorem}
\begin{proof}
Given $L,M \in G$, then $\d (\{L,M\}) = \d([s(L),s(M)])= {[\d(s(L)),\d(s(M))]}= [L,M]$. Since $\d$ is surjective we can find $l\in E$ such that $L=\d(l)$, and likewise $M=\d(m)$ and $x=\d(y)$. The Reidemeister 2 condition \eqref{R2}, which is: 
\[
\{Mx^{-1},Lx^{-1}\}\,\,\{L, x^{-1}ML^{-1}xL\}\,\,\{x^{-1} ML^{-1}x,x\}=1,
\] 
becomes \[
\{\d(my^{-1}),\d(ly^{-1})\}\,\,\{\d(l), \d(y^{-1}ml^{-1}yl)\}\,\,\{\d(y^{-1} ml^{-1}y),\d(y)\}=1,\]
or, what is the same (by using lemma \ref{bracketlift}):
\[
[my^{-1},ly^{-1}]\,\,[l, y^{-1}ml^{-1}yl]\,\,[y^{-1} ml^{-1}y,y]=1.
\] 
This is an algebraic identity which was shown to hold in the proof of Theorem \ref{eisermann}. An analogous argument shows that the Reidemeister 3 equation \eqref{R3} is satisfied, since $\d(l)\tr m = lml^{-1}$, so that we can use the algebraic identity for Reidemeister 3 from the same proof. For the Reidemeister 1 move this follows from 
\[\psi^x(M,M)=\{Mx^{-1},Mx^{-1}\}=\{\d(my^{-1}),\d(my^{-1})\}=[my^{-1},my^{-1}]=1_E .\]
\end{proof}
\begin{remark}
 {By the proof of the previous theorem, we can see that an alternative expression for $\psi^x$ is:}
\[\psi^x(L,M)=\{xML^{-1}x^{-1}Lx^{-1},Lx^{-1}\}^{-1}.\]
  {Note that for each $L,M,x \in E$ we have (since these are identities between usual commutators) that  $\{xML^{-1}x^{-1}Lx^{-1},Lx^{-1}\}^{-1}=\{L,M\} \{ ML^{-1}, x\}.$ }
\end{remark}

\subsubsection{A non-trivial example of {an} unframed  Eisermann lifting}\label{ufel}
In the context of {Theorem \ref{liftextension2},} let us find a lifting of the Eisermann invariant for the case of $G=S_5$, for which we gave detailed calculations in subsection \ref{ep}. It is well known that $S_5$ is isomorphic to $\PGL(2,5)$, the group of invertible two-by-two matrices in the field $\Z_5$, modulo the central subgroup $\Z_5^\ast$ of diagonal matrices which are multiples of the identity. We thus have a central extension:
\[\{0\} \to \Z_5^\ast \ra{i} \GL(2,5) \ra{p}  \PGL(2,5))\cong S_5\to \{1\}.\] 
Here $\GL(2,5)$ is the group of invertible two-by-two matrices in the field $\Z_5$.

Let $K_+$ and $K_-$ be the right and left handed trefoils. In  table \ref{T2} we display $\langle a |I_{\Phi^x}(K_+)| 1 \rangle$ and  $\langle a |I_{\Phi^x}(K_-)| 1 \rangle$ for some choices of $x \in \PGL(2,5)\cong S_5$, representing all possible conjugacy classes in $S_5\cong \PGL(2,5)$, in the same order as in table \ref{T}. In the {2nd} and {3rd} rows of the table,  we {put:}
 \[{\sum_{s \,\in\, \PGL(2,5)} \langle s |I_{{\Phi}^x}(K_+)| 1\rangle \quad \quad \quad \textrm{ and }\quad \quad \quad \sum_{s \,\in \,\PGL(2,5)} \langle s |I_{{\Phi}^x}(K_-)| 1\rangle,}\]
respectively, both elements of the group algebra of $\GL(2,5)$. If $A\in \GL(2,5)$, its projection to $\PGL(2,5)$ is denoted by $\widetilde{A}$. 

\begin{table}
\scriptsize
\[
{
\begin{array}{|l|l|l|l|l|l|l|l|}\hline
 x    &  \widetilde{\begin{pmatrix}  1 & 0\\ 0 &1 \end{pmatrix}}   &   \widetilde{ \begin{pmatrix}  1 & 3\\ 4 & 4  \end{pmatrix}}  &   \widetilde{\begin{pmatrix}  0 & 1\\ 1 & 0 \end{pmatrix}}       &   \widetilde{\begin{pmatrix}  4 & 1\\ 4 & 0  \end{pmatrix}}         &  \widetilde{\begin{pmatrix}  3 & 1\\ 4 &4 \end{pmatrix}  }    &   \widetilde{ \begin{pmatrix}  2 & 0\\ 0 & 1 \end{pmatrix}  }    &    \widetilde{\begin{pmatrix}  3 & 0\\ 3 & 3 \end{pmatrix} }      \\ \hline
 K_+  & \begin{pmatrix} 1 & 0\\ 0 &1 \end{pmatrix}    &   7\begin{pmatrix}  1 & 0\\ 0 &1 \end{pmatrix}   &  5 \begin{pmatrix}  1 & 0\\ 0 & 1 \end{pmatrix}     & \begin{pmatrix}  1 & 0\\ 0 & 1  \end{pmatrix}   + 6\begin{pmatrix}  4 & 0\\ 0 & 4  \end{pmatrix}           &\begin{pmatrix}  1 & 0\\ 0 & 1 \end{pmatrix} &  \begin{pmatrix}  1 & 0\\ 0 & 1 \end{pmatrix}  +4 \begin{pmatrix}  3 & 0\\ 0 & 2 \end{pmatrix}   &  \begin{pmatrix}  1 & 0\\ 0 & 1 \end{pmatrix}  +5 \begin{pmatrix}  4 & 0\\ 1 & 4 \end{pmatrix}     \\                      \hline 
 K_-  & \begin{pmatrix}  1 & 0\\ 0 &1 \end{pmatrix}   &   7\begin{pmatrix}  1 & 0\\ 0 &1 \end{pmatrix} &   5 \begin{pmatrix}  1 & 0\\ 0 & 1 \end{pmatrix}      &  \begin{pmatrix}  1 & 0\\ 0 & 1  \end{pmatrix}   + 6\begin{pmatrix}  4 & 0\\ 0 & 4  \end{pmatrix}  & \begin{pmatrix}  1 & 0\\ 0 & 1 \end{pmatrix}   &   \begin{pmatrix}  1 & 0\\ 0 & 1 \end{pmatrix}  +4 \begin{pmatrix}  2 & 0\\ 0 & 3 \end{pmatrix}   &   \begin{pmatrix}  1 & 0\\ 0 & 1 \end{pmatrix}  +5 \begin{pmatrix}  4 & 4\\ 4 & 0 \end{pmatrix}\\\hline    
\end{array}}
\]
\normalsize
\caption{\label{T2} {The lifted Eisermann invariant for the positive and negative trefoils in \ref{ufel}}}
\end{table}
\begin{table}\scriptsize
\[
\begin{array}{|l|l|l|l|l|l|l|l|}\hline
 x    &  \widetilde{\begin{pmatrix}  1 & 0\\ 0 &1 \end{pmatrix}}   &   \widetilde{ \begin{pmatrix}  1 & 3\\ 4 & 4  \end{pmatrix}}  &   \widetilde{\begin{pmatrix}  0 & 1\\ 1 & 0 \end{pmatrix}}       &   \widetilde{\begin{pmatrix}  4 & 1\\ 4 & 0  \end{pmatrix}}         &  \widetilde{\begin{pmatrix}  3 & 1\\ 4 &4 \end{pmatrix}  }    &   \widetilde{ \begin{pmatrix}  2 & 0\\ 0 & 1 \end{pmatrix}  }    &    \widetilde{\begin{pmatrix}  3 & 0\\ 3 & 3 \end{pmatrix} }      \\ \hline
 K_+  & \widetilde{\begin{pmatrix} 1 & 0\\ 0 &1 \end{pmatrix}}    &   7\widetilde{\begin{pmatrix}  1 & 0\\ 0 &1 \end{pmatrix}}   &  5 \widetilde{\begin{pmatrix}  1& 0\\ 0 & 1 \end{pmatrix} }    & 7\widetilde{\begin{pmatrix}  1 & 0\\ 0 & 1  \end{pmatrix}}           &\widetilde{\begin{pmatrix}  1 & 0\\ 0 & 1 \end{pmatrix}} &  \widetilde{\begin{pmatrix}  1 & 0\\ 0 & 1 \end{pmatrix}}  +4 \widetilde{\begin{pmatrix}  4 & 0\\ 0 & 1 \end{pmatrix}}   & \widetilde{ \begin{pmatrix}  1 & 0\\ 0 & 1 \end{pmatrix}}  +5 \widetilde{\begin{pmatrix}  3& 0\\ 3 & 3 \end{pmatrix} }    \\                      \hline 
 K_-  & \widetilde{\begin{pmatrix}  1 & 0\\ 0 &1 \end{pmatrix}}   &   7 \widetilde{\begin{pmatrix}  1 & 0\\ 0 &1 \end{pmatrix}} &   5 \widetilde{\begin{pmatrix}  1 & 0\\ 0 & 1 \end{pmatrix}}      &  7\widetilde{\begin{pmatrix}  1 & 0\\ 0 & 1  \end{pmatrix}}  & \widetilde{\begin{pmatrix}  1 & 0\\ 0 & 1 \end{pmatrix}}   &  \widetilde{ \begin{pmatrix}  1 & 0\\ 0 & 1 \end{pmatrix}}  +4 \widetilde{\begin{pmatrix}  4 & 0\\ 0 & 1 \end{pmatrix}}   & \widetilde{  \begin{pmatrix}  1 & 0\\ 0 & 1 \end{pmatrix}}  +5\widetilde{ \begin{pmatrix}  2& 0\\ 3 & 2 \end{pmatrix}}    \\ \hline
\end{array}
\]
\normalsize
\caption{\label{Ts} {The unlifted Eisermann invariant for the positive and negative trefoils in \ref{ufel}}}
\end{table}

Comparing with table \ref{Ts}, which shows the unlifted {Eisermann invariant $I_{\Phi^x_0}$ for $G={\rm PGL}(2,5)$,}  we can see that this lifting of the Eisermann invariant is strictly stronger than the Eisermann invariant itself. Specifically, looking at the penultimate column of tables \ref{T2} and \ref{Ts}, thus $x=\widetilde{ \begin{pmatrix}  2 & 0\\ 0 & 1 \end{pmatrix}  } $, we can see that the lifting distinguishes the trefoil from its mirror image. Namely, for the lifted {Eisermann invariant, noting:}  \[\widetilde{\begin{pmatrix} 3 &0\\0 &2 \end{pmatrix}}= \widetilde{\begin{pmatrix} 2 &0\\0 &3\end{pmatrix}}=\widetilde{\begin{pmatrix} 4 &0\\0 &1\end{pmatrix},}\] we have:
\[\left \langle \widetilde{\begin{pmatrix} 3 &0\\0 &2\end{pmatrix}} \Big | I_{{\Phi}^x}(K_+) \Big |  \widetilde{\begin{pmatrix} 1 & 0 \\ 0  & 1   \end{pmatrix} } \right \rangle=4\begin{pmatrix} 3 &0\\0 &2 \end{pmatrix}\neq 4\begin{pmatrix} 2 &0\\0 &3 \end{pmatrix} =\left \langle \widetilde{\begin{pmatrix} 3 &0\\0 &2\end{pmatrix}} \Big | I_{{\Phi}^x}(K_-) \Big |  \widetilde{\begin{pmatrix} 1 & 0 \\ 0  & 1   \end{pmatrix} } \right \rangle,\]                   
whereas for the unlifted Eisermann invariant $I_{\Phi^x_0}$:
\[\left \langle \widetilde{\begin{pmatrix} 3 &0\\0 &2\end{pmatrix}} \Big | I_{\Phi^x_0}(K_+) \Big |  \widetilde{\begin{pmatrix} 1 & 0 \\ 0  & 1   \end{pmatrix} } \right \rangle=4\widetilde{\begin{pmatrix} 3 &0\\0 &2 \end{pmatrix}}= \left \langle \widetilde{\begin{pmatrix} 3 &0\\0 &2\end{pmatrix}} \Big | I_{\Phi^x_0}(K_-) \Big |  \widetilde{\begin{pmatrix} 1 & 0 \\ 0  & 1   \end{pmatrix} } \right \rangle.\]                    
Table \ref{Ts} should be compared with table \ref{T}.

\subsection{Homotopy interpretation of the {liftings of} the Eisermann Invariant}
To give a homotopy interpretation of the lifting of the Eisermann invariant, Theorem \ref{liftextension2}, {we recall the notion of non-abelian tensor product and wedge product of groups, {due to Brown and Loday {\cite{BrL0,BrL}}}; see also \cite{BJR}.}
Let $G$ be a group. We define the group $G\otimes G$ (a special case of the tensor product of two groups $G\otimes H$) as being the group generated by the symbols $g\otimes h$, {where $g,h\in G$,} subject to the relations, $\forall g,h,k \in G$:
\begin{eqnarray}
 gh\otimes k & = &(ghg^{-1} \otimes gkg^{-1})\,\, (g\otimes k), \label{tp1}\\
 g\otimes hk & = & (g\otimes h)\,\, (hgh^{-1} \otimes hkh^{-1}). \label{tp2}
\end{eqnarray}

{The key fact about the {non-abelian} tensor product of groups} is that there is a homomorphism of groups $\de: G\otimes G \rightarrow G'=[G,G]$, defined on generators by $g\otimes h \mapsto [g,h]$, which is clearly surjective. Surjectivity also holds if we replace $G\otimes G$ by the group $G\wedge G$, obtained from $G\otimes G$ by imposing the {additional} relations: 
\[g\otimes g=1, \, \forall g\in G.\] 
We denote the image of $g\otimes h$ in $G\wedge G$ by $g\wedge h$. Finally \cite{BJR}, there is a left action $\bullet$ by automorphisms of $G$ on $G\otimes G$ and $G\wedge G$, given by:
\[
{g\bullet (h\otimes k) = (ghg^{-1}) \otimes (gkg^{-1}) \textrm{ and } g\bullet (h\wedge k) = (ghg^{-1}) \wedge (gkg^{-1}),}
\]
and we have two crossed modules of groups: $(\de:G\otimes G\rightarrow G', \bullet)$ and  $(\de:G\wedge G\rightarrow G', \bullet)$. This fact (which is not immediate) is Proposition 2.5 of \cite{BrL}.

The following theorem  is fully proved in \cite{BrL} and \cite{M}. {Group homology} is taken with coefficients in $\mathbb{Z}$.
\begin{Theorem}[Brown-Loday / Miller]\label{blm}
Let $G$ be a group. One has an exact sequence:
\[\{0\} \to {H}_2(G) \to G \wedge G \ra{\de} G' \to \{1\}. \]
\end{Theorem}

Consider a central extension of groups $\{0\}\to A \to E \ra{\d} G \to \{1\}$, where $G$ is finite. 
 Let $K$ be a knot. Let $C_K$ be its complement.  Then it is well known, and {follows from the asphericity of the knot complement $C_K$ \cite{Pa} {(which is therefore an Eilenberg-MacLane space)}, combined with the fact that knot complements are homology circles \cite{BZ},} that ${H}_2(\pi_1(C_K))=\{0\}$. Therefore, by {the} Brown-Loday / Miller Theorem we have:   \[\pi_1(C_K)\wedge \pi_1(C_K)\cong [\pi_1(C_K),\pi_1(C_K)]=\pi_1(C_K)',\]
canonically.  
Let now $f\colon \pi_1(C_K) \to G$ be a group morphism. 
Define \[\hat{f}=\pi_1(C_K) \wedge \pi_1(C_K) \to E,\] as acting on the generators $x \wedge  y$ of $\pi_1(C_K)\wedge \pi_1(C_K)\cong [\pi_1(C_K),\pi_1(C_K)]$ {by
$\hat{f}(x \wedge y)=\{f(y),f(x)\}^{-1}; $} see {Lemma \ref{bracketlift}}.
That the map  $\hat{f}$ respects the defining relations for the non-abelian  wedge product, follows from the fact that $\ker(\d)$ is central in $E$, {as in the proof of Theorem \ref{liftextension2}.}

Going back  to the  knot $K$, choose a base point $p \in K$. {Let} $m_p \in \pi_1(C_K)$  and $l_p \in \pi_1(C_K)$ be the associated meridian and longitude. Then \[l_p \in [\pi_1(C_K) , \pi_1(C_K)] 	\cong \pi_1(C_K) \wedge \pi_1(C_K) .\]

Given an element $ x \in G$, we thus have a knot invariant of the form:
\[ \sum_{f \colon \pi_1(C_K) \to G \textrm{ with } f(m_p)=x} \hat{f}(l_p)\in \Z[E]. \]

\begin{Theorem}
 {Given $x\in G$, { a finite group}, let $\Phi^x$ be the unframed Reidemeister pair derived from the central extension of groups $\{0\}\to A \to E \ra{\d} G \to \{1\}$; Theorem \ref{liftextension2}.
 Let $K$ be a knot, with a base point $p$. Let $L_K$ be the associated string knot. Then:}
\[\sum_{a \in G}\langle 1_G \left | I_{\Phi^x}(L_K) \right | a \rangle = \sum_{f \colon \pi_1(C_K) \to G \textrm{ with } f(m_p)=x} \hat{f}(l_p). \]
\end{Theorem}
\begin{proof}
 {The proof is exactly the same as for the unlifted case. Note Remark \ref{longformulanew}.}
\end{proof}

\section*{Acknowledgements}
J. Faria Martins was   supported by CMA/FCT/UNL, under the grant PEst-OE/ MAT/UI0297/2011. This work was partially supported by the Funda\c{c}\~{a}o para a Ci\^{e}ncia e a Tecnologia through the projects
PTDC/MAT/098770/2008, 
%Invariantes Topológicos via Geometria Diferencial 
PTDC/MAT/ 101503/2008, %Nova Geometria e Topologia 
and PEst-OE/EEI/LA0009/2013. % CAMGSD 
We would like to thank Ronnie Brown for comments.

\end{document}